\documentclass[12pt]{amsart}
\usepackage{a4wide,graphicx}
\usepackage[foot]{amsaddr}
\usepackage{color}
\usepackage{caption}
\usepackage{comment}
\usepackage{enumerate}
\usepackage[left=1in, right=1in]{geometry}
\usepackage{bbm,yfonts}
\usepackage{hyperref,amssymb}

\let\pa\partial  
  
\let\eps\varepsilon

\newcommand{\R}{{\mathbb R}} 

\newcommand{\diver}{\operatorname{div}}
\newcommand{\even}{\mathrm{e}}
\newcommand{\odd}{\mathrm{o}}

\newcommand{\T}{\mathsf T}

\newcommand{\dd}{{\mathrm{d}}}

\newcommand{\sym}{\operatorname{sym}}

\newcommand{\V}{\mathcal{V}}
\newcommand{\E}{\mathcal{E}}
\newcommand{\W}{\mathcal{W}}

\DeclareMathOperator*{\esssup}{ess\,sup}

\newcommand{\vro}{\varrho}
\newcommand{\bv}{\boldsymbol v}
\newcommand{\bu}{\boldsymbol u}
\newcommand{\bs}{\boldsymbol}
\newcommand{\bx}{\boldsymbol x}
\newcommand{\by}{\boldsymbol y}
\newcommand{\bz}{\boldsymbol z}
\newcommand{\va}{\mathsf v}
\newcommand{\ua}{\mathsf u}
\newcommand{\paa}{\mathsf p}
\newcommand{\res}{\mathsf r}
\newcommand{\bef}{\boldsymbol f}
\newcommand{\material}{\Omega_{\eps,h}}
\newcommand{\Hper}{H_{\mathrm{per}}}

\theoremstyle{plain}
\newtheorem{theorem}{Theorem}[section]   
\newtheorem{lemma}[theorem]{Lemma}   
\newtheorem{proposition}[theorem]{Proposition}  
\newtheorem{corollary}[theorem]{Corollary}  

\theoremstyle{definition}
\newtheorem{definition}{Definition}[section]

\theoremstyle{remark}
\newtheorem{remark}{Remark}[section]

\begin{document}

\title[Thin linear FSI problem]{Rigorous derivation of a linear sixth-order thin-film equation as a 
reduced model for thin fluid - thin structure interaction problems}
 
\author{Mario Bukal$^1$}
\address[1]{University of Zagreb,
Faculty of Electrical Engineering and Computing\newline
Unska 3, 10000 Zagreb, Croatia}
\email{mario.bukal@fer.hr}
\author{Boris Muha$^2$}
\address[2]{University of Zagreb,
Faculty of Science, Department of Mathematics,
Bijeni\v cka cesta 30, 10000 Zagreb, Croatia}
\email{borism@math.hr}

\thanks{This work has been supported in part by Croatian Science
Foundation under projects UIP-05-2017-7249 (MANDphy) and 3706 (FSIApp). 
Authors are very thankful to 
Igor Vel\v ci\'c (University of Zagreb) for his insightful comments and suggestions
as well as to the reviewer for his valuable remarks which improved the manuscript.}



 \begin{abstract}
We analyze a linear 3D/3D fluid-structure interaction problem between a thin layer
of a viscous fluid and a thin elastic plate-like structure with the aim of deriving 
a simplified reduced model.   
Based on suitable energy dissipation inequalities quantified in terms of two  
small parameters, thickness of the 
fluid layer and thickness of the elastic structure, we identify
the right {relation between the system coefficients and small parameters which eventually
provide a reduced model on the vanishing limit.
The reduced model is a linear sixth-order thin-film equation
describing the out-of-plane displacement
of the structure, which is justified in terms of 
weak convergence results relating its solution to the solutions of the original fluid-structure 
interaction problem.
Furthermore, approximate solutions to the fluid-structure 
interaction problem are reconstructed from the reduced model and quantitative error estimates  
are obtained, which provide even strong convergence results.} 
\end{abstract} 

\date{\today}
\maketitle

\section{Introduction}
Physical models involving fluids lubricating underneath elastic 
structures are common phenomena in nature, with ever-increasing application areas in technology.
In nature, such examples range from geophysics, like the growth of magma intrusions \cite{LPN13, Mic11}, 
the fluid-driven opening of fractures in the Earth's crust \cite{BuDe05, HBB13}, and subglacial floods \cite{DJBH08,TsRi12}, 
to biology, for instance
the passage of air flow in the lungs \cite{HHS08}, and the operation of vocal cords \cite{Tit94}.
They have also become an inevitable mechanism in industry, for example 
in manufacturing of silicon wafers \cite{HuSo02, King89} and 
suppression of viscous fingering \cite{PIHJ12, PJH14}. 
In the last two decades we witness an emergence of a huge area of 
microfluidics \cite{LBS05,HoMa04,TaVe12}
with particular applications to so called lab-on-a-chip technologies \cite{SSA04, DaFin06}, 
which revolutionized experimentations in biochemistry and biomedicine.
All those examples belong to a wider class of physical models, 
called the fluid--structure interaction 
(FSI) systems, which have recently gained a lot of attention in the applied mathematics community 
due to their important and increasing applications in medicine \cite{BGN14,BCMG16},   
aero-elasticity \cite{Bol63, CDLW16, Dow15}, marine engineering \cite{YGJ17}, etc.  

Mathematical models describing the above listed examples are coupled systems of partial differential 
equations, where fluids are typically described by the Stokes or Navier-Stokes equations, 
while structures are either described by the linear elasticity equations or by some 
lower-dimensional model, if the structure is relatively thin and has a plate-like geometry. 
If fluids are also considered to be relatively thin like in our case,
the lubrication approximation is formally employed giving rise to the Reynolds equation 
for the pressure (see e.g.~\cite{BayCha86,NazPil90}).
Coupling the Reynolds equation with the structure equation yields, after appropriate
time scaling, a reduced model given in terms of a higher-order (fourth or sixth) 
evolution equation.
Such models are common in engineering literature \cite{HoMa04,TaVe12,HBB13,LPN13} and favorable for
solving and analyzing. 
They are typically derived based on some physical assumptions, heuristic arguments, and asymptotic
expansion techniques. Despite numerous applications and abundance of the literature 
on reduced FSI models, they often lack rigorous mathematical derivation in the sense that there 
is no convergence of solutions (not even in a weak sense) of the original problem to 
solutions of the reduced problem, i.e.~the literature on the topic of rigorous derivation of 
reduced models, which we outline below, is very scarce. 

In the last twenty years there has been a lot of progress in well-posedness 
theory for the FSI problems 
(see e.g.~\cite{ALT10,BGN14,CDEM,Chu14,CSS2,DGHL03,KKLTTW18,SunBorMulti} and references within). 
Starting from various  FSI problems, \v Cani\'c, Mikeli\'c and 
others \cite{CanMik03,MikGuCan07,TamCanMik05} studied the flow through a 
long elastic axially symmetric channel and using asymptotic expansion techniques obtained several 
reduced models of Biot-type. In \cite{CanMik03} they provided a rigorous 
justification of the reduced model through a weak convergence result and the corresponding 
error estimates. In \cite{PaSt06} Panasenko and Stavre analyzed a periodic flow in thin 
channel with visco-elastic walls. The problem was initially described by a 
linear 2D (fluid)/1D (structure) FSI model, and under a special ratio of the channel 
height and the rigidity of wall a linear sixth-order thin-film equation describing the wall 
displacement emanated as the reduced model. A similar problem has been also considered in \cite{CuMP18}, 
resulting again in the reduced model described by another linear sixth-order equation. 
In both papers, reduced models have been rigorously justified by the appropriate convergence results. 
In \cite{PaSt14} Panasenko and Stavre analyzed a linear 2D/2D FSI model and using the asymptotic expansion
techniques justified the simplified 2D/1D FSI model, which was the starting point in \cite{PaSt06}.  
The study from \cite{PaSt14} has been recently generalized in \cite{PaSt20A}, where depending on different
scalings of density and rigidity of the structure, a plethora of simplified 2D/1D FSI models was justified.
Finally, in \cite{PaSt20B} Panasenko and Stavre analyze three dimensional flow in an axisymmetric thin 
tube with thin stiff elastic wall, and again depending on different scalings of density and stiffness of 
the structure, they justify various reduced 1D models.

To the best of our knowledge, rigorous derivation of a reduced 2D model starting from a simple 
linear 3D/3D FSI model, where thicknesses of both parts (fluid layer and structure) vanish 
simultaneously, is lacking in literature. 
Our aim in this paper is not only to fill this gap, but also to develop a convenient framework 
in which full understanding of the linear model will open access to rigorous derivation
of the nonlinear sixth-order thin-film equation, for instance \cite{HoMa04}, as a reduced model   
for more realistic nonlinear FSI problems. This is ongoing work \cite{BuMu20} and 
preliminaries are available in \cite{BuMu20B}. Let us summarize novelties and main contributions of our 
framework. First of all we present an {\em ansatz free} approach which is based on careful quantitative
estimates for the system's energy and energy dissipation. 
"Ansatz-free" in this context means that we have no assumptions on the shape nor size of unknowns, but only
on the system's coefficients and forcing terms. Having these estimates at hand, we 
identify the right relation between system's coefficients, which in the vanishing limit of small 
parameters provides the nontrivial reduced model given in terms of
a linear sixth-order evolution equation in 2D. Identification of the reduced model is performed rigorously
in the sense of weak convergence of solutions of the linear 3D/3D FSI problem to the solution
of the sixth-order equation, and assumptions needed for that are very weak since the weak formulation
of the linear model enjoys sufficient regularity (cf.~Theorem \ref{tm:main}). 
The relation between the system's coefficients
identifies the physical regime in which the reduced model is a good approximation of the original one.
Finally, our second main contribution are quantitative error estimates for approximation of solutions of the
original FSI problem with approximate solutions reconstructed from the reduced model. These
estimates then provide strong convergence results in respective norms (cf.~Theorem \ref{tm:EE}).

\subsection{Problem formulation}
We consider a physical model in which a three dimensional channel of relative
height $\eps > 0$ is filled
with an incompressible viscous fluid described by the Stokes equations, and the channel is covered 
by an elastic structure in the shape of a cuboid of relative height $h>0$ which is described by the linear elasticity
equations. Upon non-dimensionalisation of the model (domain and equations), 
we denote {\em (non-dimensional) material configuration domain} by 
$\material = \Omega_\eps \cup \omega\cup \Omega_h$, where 
$\Omega_\eps = (0,1)^{2}\times(-\eps,0)$ denotes the fluid domain, 
$\omega = (0,1)^{2}\times\{0\}$ is the interface between the two phases, 
which we often identify with $\omega\equiv(0,1)^2$,
and $\Omega_h = (0,1)^{2}\times(0,h)$ denotes the structure domain. 
The problem is then described by a system of partial differential equations:
\begin{align}
\vro_f\pa_t\bv - \diver\sigma_f(\bv,p) &= \bs f\,,\quad \Omega_\eps\times(0, T_\eps)\,,\label{1.eq:stokes}\\
\diver \bv &= 0\,,\quad \Omega_\eps\times(0, T_\eps)\,,\label{1.eq:divfree}\\
\vro_s\pa_{tt}\bu - \diver \sigma_s(\bu) &= 0\,, \quad \Omega_h\times(0, T_\eps)\,,\label{1.eq:elast}
\end{align}
where fluid and structure stress tensors are given respectively by
\begin{align}\label{1.eq:const_law}
\sigma_f(\bv,p) = 2\eta \sym \nabla \bv - pI_3\,,\quad 
\sigma_s(\bu) = 2\mu\sym\nabla\bu  + \lambda(\diver\bu )I_3\,. 
\end{align}
Here $\sym(\cdot)$ denotes the symmetric part of the matrix, 
$\bs f$ denotes the density of an external fluid force, and $T_\eps>0$ is a given time horizon.
Unknowns in the above system are non-dimensional quantities:
the fluid velocity $\bv$, the fluid pressure $p$, and the structure displacement $\bu$. 
Constitutive laws (\ref{1.eq:const_law}) are given in terms of non-dimensional numbers,
which are in place of physical quantities: the fluid viscosity $\eta$ and Lam\'e constants
$\mu$ and $\lambda$, while $\vro_f$ and $\vro_s$ denote non-dimensional numbers in place of
the density of the fluid and the structure, respectively. 

The two subsystems (fluid and structure) are coupled through the
interface conditions on the fixed interface $\omega$: 
\begin{align}\label{def:kinematic_bc}
\pa_t\bu &= \bv\,, \quad \omega\times(0, T_\eps)\,,
\qquad \text{(kinematic -- continuity of velocities)\,,}\\
(\sigma_f(\bv,p) - \sigma_s(\bu))\bs e_3 &=0\,,\quad \omega\times(0, T_\eps)\,,
\qquad \text{ (dynamic -- stress balance)\,.} \label{def:dynamic_bc}
\end{align}
\begin{remark}\label{rem:fixeddomain}
Contrary to the intuition of the moving interface in FSI problems, system 
(\ref{1.eq:stokes})--(\ref{def:dynamic_bc}) is posed on the fixed domain with a fixed interface.
This simplification can be seen as a linearization of truly nonlinear dynamics
under the assumption of small displacements. 
Calculations justifying these linear models in the case of fluid-plate interactions can 
be found in \cite{Bol63, KKLTTW18}. In particular, 
such models are relevant for describing the high frequency, small displacement oscillations of 
elastic structures immersed in low Reynolds number viscous fluids \cite{DGHL03}.
\end{remark}

{\em Boundary and initial conditions.} For simplicity of exposition we assume periodic boundary conditions 
in horizontal variables for all unknowns. 
On the bottom of the channel
we assume no-slip condition $\bv = 0$, and the structure is free on the top boundary, 
i.e.~$\sigma_s(u)\bs e_3 = 0$. The system is for simplicity supplemented by trivial initial conditions:
\begin{align}\label{def:InitialCond}
\bv(0)=0,\; \bu(0)=0,\; \partial_t\bu(0)=0\,.
\end{align}
\begin{remark}
Nontrivial initial conditions can also be treated in our analysis framework and under certain
assumptions the same results follow. However, for brevity of exposition we postpone this discussion
for a future work. 
We could also involve a non\-tri\-vial 
volume force on the structure (nontrivial right hand side in (\ref{1.eq:elast})) under certain 
scaling assump\-tions, similar to (A1) and (A2) below for the fluid volume forces. 
However, again for simplicity we take the trivial one, which 
is in fact a common choice for applications in microfluidics \cite{SSA04}.
\end{remark}
\begin{remark}\label{PressureDrop} The above settled framework also incorporates a physically more
relevant problem, which instead of the periodic boundary conditions, 
involves the prescribed pressure drop between the inlet and the outlet of the channel.
As described in \cite{PaSt06}, this is a matter of the right choice of the fluid volume force $\bs f$.
\end{remark}
\noindent{\em Scaling ansatz and assumptions on data.} In our analysis we will assume that small parameters $\eps$ and $h$ are
related through a power law 
\begin{enumerate}
  \item[(S1)] $\eps = h^\gamma$ for some $\gamma>0$ independent of $h$.
\end{enumerate}
 Lam\'e constants and structure 
density are also assumed to
depend on $h$ as 
\begin{enumerate}
  \item[(S2)] $\mu^h = \hat{\mu}h^{-\kappa}$, $\lambda^h = \hat{\lambda}h^{-\kappa}$ and
$\vro_s^h = \hat{\vro}_sh^{-\kappa}$ for some
$\kappa>0$,
\end{enumerate}
and $\hat{\mu},\ \hat{\lambda}$ and $\hat\vro_s$ independent of $h$.
 Finally, the time
scale of the system will be set as
\begin{enumerate}
  \item[(S3)] $\T = h^\tau$ for some $\tau\in\R$.
\end{enumerate}
Scaling ansatz of the structure data
is motivated by the fact that Lam\'e constants and density are indeed large for solid materials, 
and parameter $\kappa$ may be interpreted as a measure of stiffness of the structure material 
\cite{Cia88}. See for instance \cite{PaSt20B} for various physical examples and scaling assumptions
on the stiffness of the structure. On the other hand, fluid data, in particular fluid viscosity $\eta$ is not affected by
$\eps$ scaling and is assumed to be constant in the limiting process. This is a standard assumption
in the classical lubrication approximation theory \cite{Sze12}.

\noindent For the fluid volume force $\bs f$ we assume:

\begin{enumerate}
  \setlength\itemsep{1.5mm}
  \item[(A1)] $\|\bs f\|_{L^\infty(0,T_\eps;L^2(\Omega_\eps;\R^3))} 
+ \|\pa_{\alpha}^2\bs f\|_{L^\infty(0,T_\eps;L^2(\Omega_\eps;\R^3))} \leq C\sqrt{\eps}$ for $\alpha=1,2$,
\item[(A2)] $\|\pa_t \bs f\|_{L^2(0,T_\eps;L^2(\Omega_\eps;\R^3))}\leq C\sqrt{\eps T_\eps}$\,,
\end{enumerate}
where $C>0$ is independent of $\bef$ and $\eps$. 
\begin{remark}
(A1) is a relatively weak assumption necessary for the derivation of the energy estimate (\ref{intro:energEH}),
and consequently derivation of the reduced model (cf.~Sec.~\ref{Sec:Estimates}), 
while (A2) is mainly needed for the error estimate analysis (cf.~Sec.~\ref{sec:EE}). 
Notice also that these assumptions are not 
``small data'' assumptions, since the small factor $\sqrt{\eps}$ comes from the size of the domain.
Physically, condition (A1) means that the force is not singular in $\eps$, while (A2) means that 
it does not oscillate too much in time. Notice that $\bs f$ arising from the pressure drop 
(see Remark \ref{PressureDrop}) satisfies these assumptions, provided that the pressure drop does 
not depend on $\eps$ and does not oscillate in time, which is the case in all relevant applications.

\end{remark}

Let us emphasize at this point that unknowns of the system are \emph{ansatz free}, and our first aim
is to determine the right scaling of unknowns, which will eventually lead to a nontrivial 
reduced model as $h,\eps\downarrow0$.  
The appropriate scaling of unknowns will be determined solely from
a priori estimates, which are quantified in terms of small parameters $\eps$ and $h$.

\subsection{Main results}
The key ingredient of our convergence results, which provides all necessary a priori estimates,
is the following energy estimate. Let $(\bv^\eps,\bu^h)$ be a solution to 
(\ref{1.eq:stokes})--(\ref{def:InitialCond}), precisely defined in Section \ref{Sec:Estimates},
and assume (A1), then
\begin{align}\label{intro:energEH}
\frac{\vro_f}{2}\int_{\Omega_\eps}|\bv^\eps(t)|^2 \dd\bx
 & + \frac{\eta}{2}\int_0^t\!\!\int_{\Omega_\eps}|\nabla\bv^\eps|^2 \dd\bx\dd s 
 + \frac{\vro_s}{2}\int_{\Omega_h}|\partial_t\bu^h(t)|^2 \dd\bx \\
& + \int_{\Omega_h}\Big(\mu|\sym\nabla\bu^h(t)|^2 \nonumber
+ \frac{\lambda}{2}|\diver\bu^h(t)|^2\Big)\dd\bx \leq C t \eps^{3}\,
\end{align}
for a.e.~$t\in[0,T_\eps)$, where $C>0$ is independent of $(\bv^\eps,\bu^h)$, $\eps$, and of time
variable $t$. The proof of (\ref{intro:energEH}) is given in Section \ref{sec:impr_ee} 
(Proposition \ref{prop:impr_ee}).

Rescaling the thin domain $\Omega_{\eps,h}$ to the reference domain 
$\Omega = \Omega_-\cup\omega\cup\Omega_+$, 
as described in Section \ref{Sec:Estimates} in detail, and rescaling time and data 
according to the above scaling ansatz, the rescaled energy estimate (\ref{intro:energEH}) 
together with the weak formulation suggest to take
\begin{equation}\label{intro:tau}
\tau = \kappa - 3\gamma - 3\quad \text{and}\quad \tau \leq -1\,
\end{equation}
in order to obtain a nontrivial limit model as $h\downarrow0$.
Employing (\ref{intro:tau}) in the rescaled problem (\ref{1.eq:stokes})-(\ref{def:InitialCond}) we obtain
weak convergence results and identify the reduced FSI model. The following theorem summarizes
our first main result.
\begin{theorem}\label{tm:main}
Let $(\bv(\eps), p(\eps), \bu(h))$ be a solution to the rescaled problem of 
(\ref{1.eq:stokes})-(\ref{def:InitialCond}), then the following convergence results 
hold. For the fluid part we have
\begin{align*}
\eps^{-2}\bv(\eps)&\rightharpoonup (v_1,v_2,0)\quad\text{weakly in }L^2(0,T;L^2(\Omega_-;\R^3))\,,\\
\eps^{-2}\pa_3\bv(\eps) &\rightharpoonup (\pa_3v_1,\pa_3v_2,0)\quad\text{weakly in }L^2(0,T;L^2(\Omega_-;\R^3))\,,\\
p(\eps )&\rightharpoonup p \quad\text{weakly in }L^2(0,T;L^2(\Omega_-))\,,
\end{align*}
on a subsequence as $\eps\downarrow0$.
The limit velocities are explicitly given in terms of the pressure in the sense of distributions
\begin{equation}\label{intro:v_aplha}
v_\alpha(\by,t) = \frac{1}{2\eta}y_3(y_3+1)\pa_\alpha p(y',t) 
+ F_\alpha(\by,t) + (1+y_3)\pa_ta_\alpha(t)\,,
\quad (\by,t)\in\Omega_-\times(0,T)\,,
\end{equation} 
where $\displaystyle F_\alpha(\cdot,y_3,\cdot) 
= \frac{y_3+1}{\eta}\int_{-1}^{0} \zeta_3 f_\alpha(\cdot,\zeta_3,\cdot)\,\dd \zeta_3 + 
\frac{1}{\eta}\int_{-1}^{y_3}(y_3-\zeta_3) f_\alpha(\cdot,\zeta_3,\cdot)\,\dd \zeta_3$ 
and $\pa_ta_\alpha\in L^\infty(0,T)$ denote limit of translational structure velocities
(cf.~Section \ref{sec:translat}). 

For the structure part on the limit we find the linear bending plate model
\begin{align*}
h^{2-\kappa}\left(
\begin{array}{r}
u_1(h) - a_1(h) \\ 
u_2(h) - a_2(h)\\
 hu_3(h) 
\end{array}
\right)  
&\overset{\ast}{\rightharpoonup} 
\left(
\begin{array}{r}
-(z_3 - \frac12)\pa_1 w_3 \\[3pt] -(z_3 - \frac12)\pa_2 w_3 \\ w_3
\end{array}
\right)\quad\text{weakly* in } 
L^\infty(0,T;H^1(\Omega_+;\R^3))\,,
\end{align*}
where $w_3\in L^\infty(0,T;\Hper^2(\omega))$ and $a_\alpha(h)\subset L^\infty(0,T)$ denote
horizontal translations of the structure.
Furthermore, the vertical limit displacement $w_3$ is related to the limit pressure $p$
in the sense of distributions as
\begin{align}\label{intro:p}
p &= \chi_\tau\hat\vro_s\pa_{tt}w_3 + 
\frac{8\hat\mu(\hat\mu + \hat\lambda)}{3(2\hat\mu + \hat\lambda)}(\Delta')^2w_3\,,
\end{align} 
where $\chi_\tau = 0$ for $\tau < -1$ and $\chi_\tau = 1$ for $\tau = -1$,
 $\hat\lambda, \hat\mu$ and $\hat\vro_s$ are rescaled Lam\'e
constants and material density according to (S2), while $(\Delta')^2$ denotes the 
bi-Laplace operator in horizontal variables. Finally, the system (\ref{intro:v_aplha})--(\ref{intro:p})
is closed with a sixth-order evolution equation for $w_3$ 
\begin{equation}\label{intro:reduced}
\pa_t w_3 - \chi_\tau\frac{\hat\vro_s}{12\eta}\Delta'\pa_{tt}w_3 
- \frac{2\hat\mu(\hat\mu + \hat\lambda)}{9\eta(2\hat\mu + \hat\lambda)}(\Delta')^3w_3 = F \, 
\end{equation}
with periodic boundary conditions and trivial initial datum. The right hand side $F$ is given by
$\displaystyle F(y',t) = -\int_{-1}^0 \left( \pa_1F_1\,\dd y_3 + \pa_2 F_2\right)\dd y_3$.
\end{theorem}
\noindent We refer to equation (\ref{intro:reduced}) as a linear sixth-order thin-film equation. 
The name "thin-film equation" is consistent with the name of its more popular nonlinear siblings:
fourth-order thin-film equations \cite{BeGr05,Ber98,ODB97} and sixth-order thin-film equations
\cite{HoMa04, King89}, where nonlinearities appear due to the moving boundary of the
fluid domain. Since in our model the fluid domain is fixed (cf.~Remark \ref{rem:fixeddomain}),
depth integration of the limit divergence free equation eventually yields the linear equation
(see Section \ref{sec:thinfilmeq} below).
Complete proof of Theorem \ref{tm:main} with detailed discussions is given in Section \ref{sec:reduced}. 
 
Evolution equation (\ref{intro:reduced}) now serves as a reduced FSI model of the original
problem (\ref{1.eq:stokes})--(\ref{def:InitialCond}). Namely, by solving (\ref{intro:reduced}), we can
approximately reconstruct solutions of the original FSI
problem in accordance with the convergence results of the previous 
theorem. Let $w_3$ be a solution of equation (\ref{intro:reduced}). The 
approximate pressure $\paa^\eps$ is defined by 
\begin{align*}
\paa^\eps(\bx,t) = p(x',t)\,,\quad (\bx,t)\in\Omega_\eps\times(0,T)\,,
\end{align*}
where $p$ is given by (\ref{intro:p}) and the approximate fluid velocity $\bs\va^\eps$ is defined by
\begin{align*}
{\bs\va}^\eps(\bx,t)=\eps^2\Big(v_1(x',\frac{x_3}{\eps},t),v_2(x',\frac{x_3}{\eps},t),0\Big)\,,
\quad (\bx,t)\in\Omega_\eps\times(0,T)\,,
\end{align*}
with $v_\alpha$ given by (\ref{intro:v_aplha}). Accordingly, we also define the
approximate displacement $\bs\ua^h$ as
\begin{equation*}
\bs\ua^h(\bx,t) = h^{\kappa-3}\left(h^{-\gamma}a_1 - \Big(x_3 
- \frac{h}{2}\Big)\pa_1w_3(x',t),h^{-\gamma}a_2 - \Big(x_3 - \frac{h}{2}\Big)\pa_2w_3(x',t), 
w_3(x',t) \right)\,
\end{equation*}
for all $(\bx,t)\in\Omega_h\times(0,T)\,,$ where 
$\displaystyle a_\alpha(t) = \int_0^t\pa_ta_\alpha \dd s$, $\alpha = 1,2$, and 
$\pa_ta_\alpha$ are given by (\ref{def:pata}).    
Observe that approximate solutions are defined
on the original thin domain $\Omega_{\eps,h}$, but in rescaled time.

Our second main result provides error estimates for approximate solutions, and thus strong convergence
results in respective norms.
\begin{theorem}\label{tm:EE}
Let $(\bv^\eps,p^\eps,\bu^h)$ be a solution to the original FSI problem 
(\ref{1.eq:stokes})--(\ref{def:InitialCond}) in rescaled time and  
let $(\bs\va^\eps,\paa^\eps,\bs\ua^h)$ be approximate solution constructed from the reduced model
as above. Let us additionally assume that 
$\max\{2\gamma+1, \frac74\gamma+\frac32\} \leq \kappa < 2 + 2\gamma$, then
\begin{align*}
\|\bv^\eps - \bs \va^\eps\|_{L^2(0,T;L^2(\Omega_\eps))} &\leq 
C\eps^{5/2}h^{\min\{\gamma/2,\,2\gamma-\kappa+2\}}\,,\\
\|p^\eps - \paa^\eps\|_{L^2(0,T;L^2(\Omega_\eps))} 
&\leq C\eps^{1/2}h^{\min\{\gamma/2,2\gamma-\kappa+2\}}\,,\\
\|u_\alpha^h - \ua_\alpha^h\|_{L^\infty(0,T;L^2(\Omega_h))} &\leq 
Ch^{\kappa-3/2}h^{\min\{1,\gamma/2, 2\gamma+2-\kappa\}} + 
C\sqrt{h}\|a_\alpha^h - h^{\kappa-3-\gamma}a_\alpha\|_{L^\infty(0,T)}\,,\\
\|u_3^h - \ua_3^h\|_{L^\infty(0,T;L^2(\Omega_h))} &\leq 
Ch^{\kappa-5/2}h^{\min\{1/2,\gamma/2, 2\gamma+2-\kappa\}}\,,
\end{align*}
where $C>0$ denote generic positive constants independent of $\eps$ and $h$.
\end{theorem}

\begin{remark}
Note that the error estimate of horizontal fluid velocities relative to the norm of velocities, as 
well as the relative error estimate of the pressure is
 $O(h^{\min\{\gamma/2, 2\gamma-\kappa+2\}})$. Hence, for $\kappa\leq\frac32\gamma+2$, 
 this convergence rate is $O(\sqrt{\eps})$. Since $v_3^\eps$ is of lower order, we would need to
 construct a better (higher-order) corrector for establishing error estimates in the vertical 
 component of the fluid velocity. Such construction would require additional tools and would thus
 exceed the scope of this paper. In the leading order of the structure displacement, the
 vertical component, we have the relative convergence rate $O(h^{\min\{1/2,\gamma/2,2\gamma-\kappa+2\}})$,
which for $\kappa\leq \frac32\gamma+2$ means $O(h^{\min\{1/2,\gamma/2\}})$, i.e.~$O(\sqrt{\eps})$ for $\gamma\leq 1$ and
$O(\sqrt{h})$ for $\gamma>1$. In horizontal components, in-plane displacements, the dominant part
of the error estimates are errors in horizontal translations, which are artefacts of
periodic boundary conditions (cf.~Section \ref{sec:translat}). Neglecting these errors which cannot be controlled in a better way, 
the relative error estimate of horizontal displacements is 
$O(h^{\min\{1,\gamma/2, (2\gamma+2-\kappa)_+\}})$. For $\kappa\leq\frac32\gamma+2$, this estimate is 
$O(h^{\min\{1,\gamma/2\}})$, which in addition means $O(\sqrt{\eps})$ for $\gamma\leq2$ and
$O(h)$ for $\gamma > 2$.  

Let us point out that one cannot expect better convergence rates for such first-order
approximation without dealing with boundary layers, which arise around interface $\omega$
due to a mismatch of the interface conditions for approximate solutions.
For example, in \cite{MaPa01} the obtained convergence rate for the Poiseuille flow in the case 
of rigid walls of the fluid channel is $O(\sqrt{\eps})$. On the other hand, convergence rate for
the clamped Kirchhoff-Love plate is found to be $O(\sqrt{h})$ \cite{Des81}.
Additional conditions on parameters $\kappa$ and $\gamma$ which appear in the theorem
are mainly due to technical difficulties of dealing with structure translations in horizontal
directions. If these translations were not present in the model, the error estimates of 
Theorem \ref{tm:EE} would improve.
\end{remark}
\noindent The proof of Theorem \ref{tm:EE} is demonstrated in Section \ref{sec:EE}.

\section{Energy estimates and weak solutions}\label{Sec:Estimates}
\subsection{Notation and definitions}
Let $\bx=(x',x_3)=(x_1,x_2,x_3)\in \Omega_{\eps,h}$ denotes the spatial variables in the original thin 
domain and let $\by=(y',y_3)=(x',x_3/\eps)\in (0,1)^{2}\times (-1,0) =:\Omega_-$ 
and $\bz=(z',z_3)=(x',x_3/h)\in (0,1)^{2}\times (0,1)=:\Omega_+$ denote the fluid and the structure variables in the reference domain, respectively. A sketch of the original
thin domain is depicted in Figure \ref{fig:domain}.    
\begin{figure}
\includegraphics[width = 8cm]{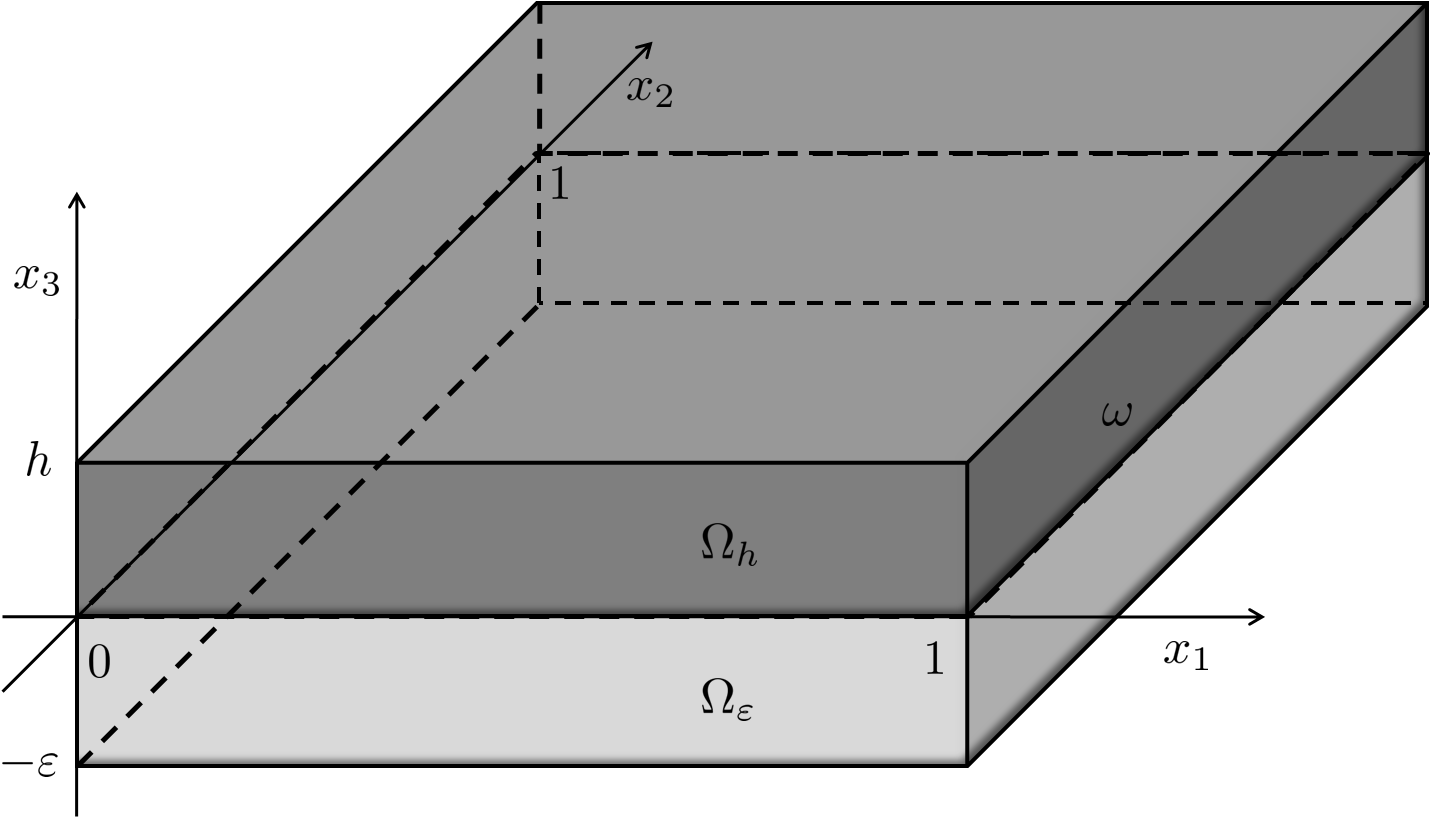}
\caption{Sketch of the original thin domain $\Omega_{\eps,h}$.}
\label{fig:domain}
\end{figure}
Solutions in the original domain $\material$ will be denoted by $\eps$ or $h$ in superscripts,
i.e.~$\bv^\eps$, $p^\eps$ and $\bu^h$. On the reference domain, solutions 
will be denoted by $\eps$ or $h$ in parentheses and they are defined according to
\begin{equation}\label{def:UepsH}
\bv(\eps)(\by,t):=\bv^\eps(\bx,t)\,,\quad p(\eps)(\by,t):=p^\eps(\bx,t)\,,\quad \bu(h)(\bz,t):=\bu^h(\bx,t)
\end{equation}
for all $(\bx,t)\in\Omega_{\eps,h}\times(0,T_\eps)$.
As standard, vectors and vector-valued functions are written in bold font. The inner product between two
vectors in $\R^3$ is denoted by one dot $\cdot$ and the inner product between two matrices is denoted by
two dots $:$\,. 
Next, we denote scaled gradients by
$\nabla_\eps=(\partial_{y_1},\partial_{y_2},\dfrac{1}{\eps}\partial_{y_3})$ and 
$\nabla_h=(\partial_{z_1},\partial_{z_2},\dfrac{1}{h}\partial_{z_3})$, and they satisfy 
the following identities 
\begin{equation}\label{def:nablaEpsH}
\nabla\bv^{\eps}(\bx,t)=\nabla_\eps \bv(\eps)(\by,t),\qquad \nabla \bu^{h}(\bx,t)=\nabla_h\bu(h)(\bz,t)\,,
\quad (\bx,t)\in\Omega_{\eps,h}\times(0,T_\eps)\,.
\end{equation}
When the domain of a function is obvious, partial derivatives $\pa_{x_i}$, $\pa_{y_i}$ 
or $\pa_{z_i}$ will be 
simply denoted by $\pa_i$ for $i=1,2,3$. Greek letters $\alpha,\beta$ in indices will 
indicate only horizontal variables, i.e.~$\alpha,\beta = 1,2$.

The basic energy estimate for the original FSI problem (\ref{1.eq:stokes})--(\ref{def:InitialCond}), 
given in Section \ref{sec:basic_ee} below,
suggest the following functions spaces as appropriate for the definition of 
weak solutions and test functions. 
For fluid velocity, the appropriate function space appears to be
\begin{align*}
\quad \V_F(0,T_\eps;\Omega_\eps) = L^\infty(0,T_\eps;L^2(\Omega_\eps;\R^3))\cap L^2(0,T_\eps;V_F(\Omega_\eps))\,,
\end{align*}
where $V_F(\Omega_\eps) = \left\{\bv\in H^1(\Omega_\eps;\R^3)\ 
:\ \diver\bv = 0\,,\ \bv|_{x_3=-\eps}=0\,,\ \bv \text{ is }\omega\text{-periodic} \right\}$, and
$T_\eps>0$ is a given time horizon. Even though we will work with global-in-time solutions, 
all estimates will be carried out on the time interval $(0,T_{\eps})$ in accordance with ansatz (S3). 
Here the notation $T_{\eps}$ is introduced to emphasize the difference between the physical 
time-horizon $T_{\eps}$ used in this section and the re-scaled time horizon $T$ used in later sections.
Similarly, the structure function space will be
\begin{align*}
\V_S(0,T_\eps;\Omega_h) = W^{1,\infty}(0,T_\eps;L^2(\Omega_h;\R^3))\cap L^\infty(0,T_\eps;V_S(\Omega_h))\,,
\end{align*}
where $V_S(\Omega_h) = \left\{\bu\in H^1(\Omega_h;\R^3)\ :\ \bu \text{ is }\omega\text{-periodic} \right\}$.
Finally, the solution space of the coupled problem (\ref{1.eq:stokes})--(\ref{def:InitialCond}) on the 
thin domain will be compound of previous spaces involving the kinematic interface condition
(\ref{def:kinematic_bc}) as a constraint:
\begin{align}\label{2.def:sol_space}
\V(0,T_\eps;\Omega_{\eps,h}) = \big\{(\bv,\bu)\in \V_F(0,T_\eps;\Omega_\eps)&\times \V_S(0,T_\eps;\Omega_h)\ :\\
\bv(t) &= \pa_t\bu(t)\text{ on }\omega\text{ for a.e. }t\in(0,T_\eps)\big\}.\nonumber
\end{align}
Now we can state the definition of weak solutions to our problem in the sense of Leray and Hopf.
\begin{definition}
We say that a pair $(\bv^\eps,\bu^h)\in \V(0,T_\eps;\Omega_{\eps,h})$ is a \emph{weak solution}
to the linear FSI problem (\ref{1.eq:stokes})--\eqref{def:InitialCond}, 
if the following variational equation holds in $\mathcal{D}'(0,T_\eps)$:
\begin{align}
&\vro_f\frac{\dd}{\dd t}\int_{\Omega_\eps}\bv^\eps\cdot\bs\phi\,\dd \bx
-\vro_f\int_{\Omega_\eps}\bv^\eps\cdot\pa_t\bs\phi\,\dd \bx \nonumber
  + 2\eta\int_{\Omega_\eps}\sym\nabla\bv^\eps:\sym\nabla\bs\phi\,\dd \bx \\
&\vro_s\frac{\dd}{\dd t}\int_{\Omega_h}\bu^h\cdot\pa_t\bs\psi\,\dd \bx 
- \vro_s\int_{\Omega_h}\pa_t\bu^h\cdot\pa_t\bs\psi\,\dd \bx  + \label{eq:weak_fsi}
\int_{\Omega_h}(2\mu\sym\nabla\bu^h:\sym\nabla\bs\psi 
\\ & \qquad + \lambda\diver\bu^h\diver\bs\psi)\,\dd \bx  
= \int_{\Omega_\eps}\bs f^\eps\cdot\bs\phi \,\dd \bx
\nonumber
\end{align}
for all $(\bs\phi,\bs\psi)\in \W(0,T_\eps;\Omega_{\eps,h})$, where
\begin{align*}
\W(0,T_\eps;\Omega_{\eps,h}) = \big\{ (\bs\phi,\bs\psi)\in 
C^1\big([0,T_\eps]; & V_F(\Omega_\eps)\times V_S(\Omega_h)\big)\, :\\
& \bs\phi(t) = \bs\psi(t) \text{ on }\omega \text{ for all }t\in [0,T_\eps] \big\}\,
\end{align*}
denotes the space of test functions. Moreover, $(\bv^\eps,\bu^h)$ verify the energy dissipation
inequality (\ref{eq.energy_basic}) given below.
\end{definition}

\subsection{Auxiliary inequalities on thin domains}

In the next proposition we collect a few important functional inequalities, which will be frequently used  
in the subsequent analysis.
\begin{proposition}\label{Poincare} 
Let $0 < \eps \ll 1$ and $\bv^\eps\in V_F(\Omega_{\eps})$, then the following inequalities hold:
\begin{align}
& \|\bv^\eps\|_{L^2(\Omega_\eps)}\leq C\eps\|\partial_3\bv^\varepsilon\|_{L^2(\Omega_\eps)}\,, \quad\text{(Poincar\'e inequality)}\,,
\label{2.ineq:eps_P}\\
& \|\bv^\eps\|_{L^2(\omega)}\leq C\sqrt{\eps}\|\partial_3\bv^\eps\|_{L^2(\Omega_\eps)}\,, \quad \text{(Trace inequality)}\,,
\label{2.ineq:eps_T}\\
&  
\|\partial_{\alpha} v_3^\varepsilon\|_{L^2(\Omega_\eps)}
\leq \dfrac{C}{\eps} 
\|\sym\nabla \bv^\eps\|_{L^2(\Omega_\eps)}\,,\ \  \alpha=1,2\,,
\quad \text{(Korn inequality)}\,.
\label{2.ineq:eps_K} 
\end{align}
All above constants $C$ are positive and independent of $\eps$.
\end{proposition}
\proof
Utilizing the Cauchy-Schwarz inequality, we calculate
\begin{align*}
\|\bv^\eps\|_{L^2(\Omega_\eps)}^2 = \int_{\omega}\int_{-\eps}^0\bv^\eps(x',x_3)^2\dd x'\dd x_3
=\int_{\omega}\int_{-\eps}^0\Big (\int_{-\eps}^{x_3}\partial_3\bv^\eps(x',s)\dd s\Big )^2\dd x'\dd x_3
\\
\leq \int_{\omega}\int_{-\eps}^0 (x_3+\eps)\int_{-\eps}^{x_3}(\partial_3\bv^\eps)^2(x',s)\dd s
\leq \int_{\omega}\int_{-\eps}^0 (x_3+\eps)\int_{-\eps}^0(\partial_3\bv^\eps)^2(x',s)\dd s
\\
\leq \|\partial_3\bv^\eps\|_{L^2(\Omega_\eps)}^2\int_{-\eps}^0 (x_3+\eps)\dd x_3
=\frac32\eps^2\|\partial_3\bv^\eps\|_{L^2(\Omega_\eps)}^2\,,
\end{align*}
which proves the Poincar\'e inequality (\ref{2.ineq:eps_P}). 

Similar calculations with application of the Jensen's inequality give:
\begin{align*}
\|\bv^\eps\|_{L^2(\omega)}^2 = \int_w |\bv^\eps(x',0)|^2\dd x'
=\int_w\eps^2\Big |\frac{1}{\eps}\int_{-\eps}^0\pa_3\bv^\eps(\bx)\dd x_3 \Big |^2\dd x'
\leq \varepsilon\int_{\Omega_\eps}|\pa_3\bv^\eps(\bx)|^2\dd \bx\,,
\end{align*}
which proves the trace inequality (\ref{2.ineq:eps_T}).

Finally, the Korn inequality (\ref{2.ineq:eps_K}) follows directly from Theorem \ref{app:thinKorn}, 
and boundary condition $\bv^\eps=0$ on the bottom part of the fluid domain $\{x_3=-\varepsilon\}$. 
\qed

\subsection{Basic energy estimate} \label{sec:basic_ee}
First we derive a basic energy estimate quantified only in terms of the relative 
fluid thickness $\eps$. 
\begin{proposition}
Let us assume (A1) and let $(\bv^\eps,\bu^h)\in \V(0,T_\eps;\Omega_{\eps,h})$ be a solution 
to (\ref{eq:weak_fsi}). There 
exists a constant $C>0$, independent of $\eps$ and $T_\eps$, such that
the following energy estimate holds
\begin{align}
\frac{\vro_f}{2}\int_{\Omega_\eps}|\bv^\eps(t)|^2 \dd \bx
&+ \eta\int_0^t\!\!\int_{\Omega_\eps}|\sym\nabla\bv^\eps(s)|^2 \dd\bx \dd s 
+ \frac{\vro_s}{2}\int_{\Omega_h}|\pa_t\bu^h(t)|^2\dd\bx \nonumber\\
&+ \mu \int_{\Omega_h}|\sym\nabla\bu^h(t)|^2\dd\bx
+ \frac{\lambda}{2}\int_{\Omega_h}|\diver \bu^h(t)|^2\dd\bx \label{eq.energy_basic}
\leq Ct\eps\,,
\end{align}
for a.e.~$t\in[0,T_\eps)$.
\end{proposition}
\begin{proof}
Here we present just a formal argument for the basic energy estimate which can be made rigorous 
in the standard way by using Galerkin approximations and the weak lower semicontinuity of the energy
functional, see e.g.~\cite{Galdi}. 
Let us take $(\bv^\eps,\partial_t \bu^h)$ as test functions in \eqref{eq:weak_fsi},
then straightforward calculations and integration in time from $0$ to $t$ gives
\begin{align}
\frac{\vro_f}{2}\int_{\Omega_\eps}|\bv^\eps(t)|^2 \dd \bx
&+ 2\eta\int_0^t\!\!\int_{\Omega_\eps}|\sym\nabla\bv^\eps(s)|^2 \dd\bx \dd s 
+ \frac{\vro_s}{2}\int_{\Omega_h}|\pa_t\bu^h(t)|^2\dd\bx \nonumber\\
&+ \mu \int_{\Omega_h}|\sym\nabla\bu^h(t)|^2\dd\bx
+ \frac{\lambda}{2}\int_{\Omega_h}|\diver \bu^h(t)|^2\dd\bx \label{eq.energy0}
 = \int_0^t\!\!\int_{\Omega_\eps}\bs f^\eps\cdot\bv^\eps \dd\bx\dd s\,.
\end{align}
Now let us estimate the right hand side. 
First, applying the Cauchy-Schwarz inequality, then employing the assumption (A1) on the volume 
force $\bs f^\eps$, and utilizing Poincar\'e and Korn inequalities 
from Proposition \ref{Poincare}, we obtain respectively,
\begin{align}
\left|\int_0^t\!\!\int_{\Omega_\eps}\bef^\eps\cdot\bv^\eps \dd\bx\dd s\right|
&\leq \int_0^t\|\bef^\eps\|_{L^2(\Omega_\eps)}\|\bv^\eps\|_{L^2(\Omega_\eps)}\dd s
\leq C\int_0^t  \sqrt{\eps} \eps\|\partial_3\bv^\eps\|_{L^2(\Omega_\eps)}\dd s\nonumber\\
&\leq Ct \eps + \eta\int_0^t 
\|\sym\nabla \bv^\eps\|_{L^2(\Omega_\eps)}^2\dd s\,.\label{3.ineq:rhs1}
\end{align}
The latter inequality is obtained by choosing a suitable constant in the application of the 
Young inequality such that the last term can be absorbed in the left-hand side of (\ref{eq.energy0}),
which finishes the proof.
\end{proof}


\subsection{Existence and regularity of weak solutions} 
Although (\ref{1.eq:stokes})--(\ref{def:InitialCond}) is a linear problem, the existence
analysis is not trivial. Well-posedness for related (but geometrically different) problem has been 
first established in \cite{DGHL03} using a Galerkin approximation scheme, and later in 
\cite{AvTr07, AvTr09} using the semigroup approach.
The existence analysis for a 2D/2D analogue of (\ref{1.eq:stokes})--(\ref{def:InitialCond}) 
has been performed in \cite{PaSt14} using the Galerkin approximation scheme, while the
regularity issues have been completely resolved in \cite{ALT10}.
Straightforward extension of these results from \cite{PaSt14} and \cite{ALT10} to 
problem (\ref{1.eq:stokes})--(\ref{def:InitialCond})
yields the following proposition. 
\begin{proposition}\label{prop:exreg}
Let $T_\eps>0$ be given and let assumption (A1) holds. 
There exists a unique solution $(\bv^\eps,\bu^h)\in \V(0,T_\eps;\Omega_{\eps,h})$
to (\ref{eq:weak_fsi}), which additionally satisfies:
\begin{enumerate}[(a)]
  \item (time regularity) \begin{equation*}
\pa_t \bv^\eps\in L^\infty(0,T_\eps;L^2(\Omega_{\eps};\R^3))\quad\text{and}\quad
\pa_{tt}\bu^h\in L^\infty(0,T_\eps;L^2(\Omega_h;\R^3))\,,
\end{equation*}
\item (space regularity)
\begin{align*}
\bv^\eps\in L^\infty(0,T_\eps;H^2(\Omega_\eps;\R^3))\quad\text{and}\quad 
\bu^h\in L^\infty(0,T_\eps;H^2(\Omega_h;\R^3))\,.
\end{align*}
\end{enumerate}
Moreover, there exists a unique pressure $p^\eps\in L^2(0,T;H^1(\Omega_\eps))$ such that $(\bv^\eps,p^\eps,\bu^h)$
solves the original problem (\ref{1.eq:stokes})--(\ref{def:InitialCond}) in the classical sense.
\end{proposition}


\subsection{Improved energy estimates}\label{sec:impr_ee} 
Next, we aim to improve the basic energy estimate (\ref{eq.energy_basic}).
\begin{proposition}\label{prop:impr_ee}
Assume that the volume force $\bs f^\eps$ satisfies (A1) and 
let $(\bv^\eps,\bu^h)\in \V(0,T_\eps;\Omega_{\eps,h})$ be the solution to (\ref{eq:weak_fsi}). There 
exists a constant $C>0$, independent of $\eps$ and $T_\eps$, such that
the following energy estimate holds
\begin{align}\label{ineq:energEH}
\frac{\vro_f}{2}\int_{\Omega_\eps}|\bv^\eps(t)|^2 \dd\bx
 & + \frac{\eta}{2}\int_0^t\!\!\int_{\Omega_\eps}|\nabla\bv^\eps|^2 \dd\bx\dd s 
 + \frac{\vro_s}{2}\int_{\Omega_h}|\partial_t\bu^h(t)|^2 \dd\bx \\
& + \int_{\Omega_h}\Big(\mu|\sym\nabla\bu^h(t)|^2 \nonumber
+ \frac{\lambda}{2}|\diver\bu^h(t)|^2\Big)\dd\bx \leq C t \eps^{3}\,
\end{align}
for a.e.~$t\in[0,T_\eps)$.
\end{proposition} 
\noindent Observe that in (\ref{ineq:energEH}), unlike in (\ref{eq.energy_basic}), we control 
the full gradient of the fluid velocity and the estimate is quantitatively improved in terms of $\eps$.
\begin{proof}
Let us formally take $(-\partial_{\alpha}^2\bv^\eps,-\partial_{\alpha}^2\partial_t\bu^h)$ for
$\alpha = 1,2$, as test functions in \eqref{eq:weak_fsi}. After integrating by parts in horizontal variables
and assuming (A1), the right-hand side can 
be estimated using the Poincar\'e and Korn inequalities from Proposition \ref{Poincare}, 
as well as the basic energy inequality \eqref{eq.energy_basic} as follows:
\begin{align*}
\left|\int_0^t\!\!\int_{\Omega_\eps}\pa_{\alpha}^2\bef^\eps\cdot\bv^\eps\dd \bx\dd s\right|
&\leq C\int_0^t\!\! \eps^{3/2}\|\partial_3\bv^\eps\|_{L^2(\Omega_\eps)} \dd s
\leq C\int_0^t\!\! \eps^{1/2} 
\|\sym\nabla \bv^\eps\|_{L^2(\Omega_\eps)} \leq C t \eps\,.
\end{align*}
Therefore, for a.e.~$t\in[0,T_\eps)$ we have
\begin{align}
&\frac{\vro_f}{2}\int_{\Omega_\eps}|\partial_\alpha\bv^\eps(t)|^2\dd \bx 
+ 2\eta\int_0^t\!\!\int_{\Omega_\eps}|\sym\nabla\partial_\alpha\bv^\eps(s)|^2 \dd \bx\dd s
+ \frac{\vro_s}{2}\int_{\Omega_h}|\pa_t\pa_\alpha\bu^h(t)|^2\dd \bx  \nonumber\\
& \quad + \mu \int_{\Omega_h}|\sym\nabla\partial_\alpha\bu^h(t)|^2\dd \bx
+ \frac{\lambda}{2}\int_{\Omega_h}|\diver \partial_\alpha\bu^h(t)|^2\dd \bx \label{eq:energy_estimate3}
\leq  Ct \eps\,,\quad \alpha=1,2\,.
\end{align}
The above formalism for weak solutions can again be justified by standard arguments 
using finite difference quotients instead of partial derivatives (see e.g \cite{GilTru}).

Next, we invoke the identity: for every $t\in (0,T_\eps)$ 
\begin{align*}
\|\sym \nabla\bv^\eps\|^2_{L^2(0,t;L^2(\Omega_\eps))}
=\frac12\|\nabla \bv^\eps\|^2_{L^2(0,t;L^2(\Omega_\eps))} 
+ \frac12\int_0^t\!\!\int_{\Omega_\eps}\nabla\bv^\eps:\nabla^T\bv^\eps\dd\bx\dd s\,.
\end{align*}
Integrating by parts, using corresponding boundary conditions, 
and employing the divergence free equation
the second term on the right hand side can be simplified to
\begin{align}\label{eq.skw_v}
\int_0^t\!\!\int_{\Omega_\eps}\nabla\bv^\eps:\nabla^T\bv^\eps \dd\bx\dd s 
=2\int_0^t\!\!\int_{\omega}(v_1^\eps\partial_1 v_3^\eps + v_2^\eps\pa_2v_3^\eps)\dd x'\dd s\,.
\end{align}
Taking $(\bv^\eps,\pa_t\bu^h)$ as a
test function in \eqref{eq:weak_fsi} and using the identity (\ref{eq.skw_v}), 
we find
\begin{align} 
\frac{\vro_f}{2}\int_{\Omega_\eps}|\bv^\eps(t)|^2 \dd\bx
 & + \eta\int_0^t\!\!\int_{\Omega_\eps}|\nabla\bv^\eps|^2 \dd\bx\dd s \nonumber \\
+ \frac{\vro_s}{2}\int_{\Omega_h}|\partial_t\bu^h(t)|^2 \dd\bx \label{EnergijaBolja}
& + \int_{\Omega_h}\Big(\mu|\sym\nabla\bu^h(t)|^2
+ \frac{\lambda}{2}|\diver\bu^h(t)|^2\Big)\dd\bx \\ \nonumber
& = \int_0^t\!\!\int_{\Omega_\eps}\bef^\eps\cdot\bv^\eps \dd\bx\dd s
- 2\eta\int_0^t\!\!\int_{\omega}(v_1^\eps\partial_1 v_3^\eps + v_2^\eps\pa_2v_3^\eps)\dd\bx\dd s\,, 
\end{align}
for a.e.~$t\in[0,T_\eps)$.
The force term is again estimated in a similar fashion like in \eqref{3.ineq:rhs1}:
\begin{align}\label{ineq:force_control}
\left|\int_0^t\!\!\int_{\Omega_\eps}\bef^\eps\cdot\bv^\eps \dd\bx\dd s\right| 
\leq \int_0^t\!\!\eps^{3/2}\|\partial_3\bv^\eps\|_{L^2(\Omega_\eps)}\dd s
\leq Ct\eps^{3} + \frac{\eta}{4}\int_0^t\!\!\|\partial_3\bv^\eps\|_{L^2(\Omega_\eps)}^2\dd s\,,
\end{align}
but now controling the full gradient of $\bv^\eps$, which provides a better estimate in terms of $\eps$.

The interface terms in \eqref{EnergijaBolja} are estimated in the following way, separately for every 
$\alpha= 1,2$. First, using the Cauchy-Schwarz inequality and the trace inequality from
Proposition \ref{Poincare}, we obtain
\begin{align}
\left|\int_0^t\!\!\int_{\omega}v_\alpha^\eps\partial_\alpha v_3^\eps \dd\bx\dd s\right|
\leq \int_0^t\!\!\|v_\alpha^\eps\|_{L^2(\omega)}\|\partial_\alpha v_3^\eps\|_{L^2(\omega)}\dd s
\leq C\eps\int_0^t \|\partial_3\bv^\eps\|_{L^2(\Omega_\eps)}
\|\partial_3\partial_\alpha v_3^\eps\|_{L^2(\Omega_\eps)}\dd s\,.\nonumber
\end{align}
Since the term $\partial_3\partial_\alpha v_3^\eps$ is a diagonal element 
of $\sym\nabla\pa_\alpha\bv^\eps$, according to (\ref{eq:energy_estimate3}), we further estimate
\begin{align}\label{est:interface_terms} 
\left|\int_0^t\!\!\int_{\omega}v_\alpha^\eps\partial_\alpha v_3^\eps \dd\bx\dd s\right| 
\leq C t^{1/2} \eps^{3/2}\|\partial_3\bv^\eps\|_{L^2(0,t;L^2(\Omega_\eps))}
\leq Ct\eps^{3} + \frac{\eta}{4}\int_0^t\!\!\|\partial_3\bv^\eps\|_{L^2(\Omega_\eps)}^2\dd s\,.
\end{align}
Going back to (\ref{EnergijaBolja}) 
we conclude the improved energy estimate (\ref{ineq:energEH}).
\end{proof}

Assuming additional regularity of solutions and repeating formally the above arguments,
one obtains improved higher-order energy estimates, which will be used in estimating the
pressure and later in the error analysis in Section \ref{sec:EE}.

\begin{corollary}\label{TimeRefEst}
Let us assume that (A2) holds 
and let 
$(\bv^\eps,\bu^h)\in \V(0,T_\eps;\Omega_{\eps,h})$ be the solution to (\ref{eq:weak_fsi}). There 
exists a constant $C>0$, independent of $\eps$ and $T_\eps$, such that
the following a priori estimates hold:
\begin{align}\label{ineq:energEHH}
\frac{\vro_f}{2}\int_{\Omega_\eps}|\pa_t\bv^\eps(t)|^2 \dd\bx
  &+ \frac{\eta}{2}\int_0^t\!\!\int_{\Omega_\eps}|\nabla\pa_t\bv^\eps|^2 \dd\bx\dd s   
 + \frac{\vro_s}{2}\int_{\Omega_h}|\partial_{tt}\bu^h(t)|^2 \dd\bx \\
& + \int_{\Omega_h}\Big(\mu|\sym\nabla\pa_t\bu^h(t)|^2 \nonumber
+ \frac{\lambda}{2}|\diver\pa_t\bu^h(t)|^2\Big)\dd\bx 
\leq C\,\T^{-1} \eps^{3}\\ \label{ineq:energEHHS}
\frac{\vro_f}{2}\int_{\Omega_\eps}|\pa_\alpha\bv^\eps(t)|^2 \dd\bx
  &+ \frac{\eta}{2}\int_0^t\!\!\int_{\Omega_\eps}|\nabla\pa_\alpha\bv^\eps|^2 \dd\bx\dd s 
 + \frac{\vro_s}{2}\int_{\Omega_h}|\partial_t\pa_\alpha\bu^h(t)|^2 \dd\bx \\
& + \int_{\Omega_h}\Big(\mu|\sym\nabla\pa_\alpha\bu^h(t)|^2 \nonumber
+ \frac{\lambda}{2}|\diver\pa_\alpha\bu^h(t)|^2\Big)\dd\bx \leq C t \eps^{3}\,
\end{align}
for a.e.~$t\in[0,T_\eps)$ and $\alpha=1,2$.
\end{corollary}

\subsection{Rigid body displacements}\label{sec:translat}
Since the boundary conditions for the structure equations are periodic on the lateral boundaries 
and only stress in prescribed on the interface and upper boundary, 
the structure is not anchored and nontrivial rigid body displacements arise 
as part of solutions. 
However, the periodic boundary conditions prevent rotations, and due to the coupling with the fluid,
translations can also be controlled.
First, the kinematic coupling in the vertical direction 
together with the incompressibility of the fluid imply
$$
\frac{\dd}{\dd t}\int_{\omega}u_3^h(t)\,\dd x'=\int_{\omega}\partial_t u_3^h(t)\,\dd x'
=\int_{\omega}v_3^{\eps}(t)\,\dd x'=\int_{\Omega_{\eps}}\diver\bv^{\eps}(t)\dd \bx=0\,.
$$
Therefore, due to the trivial initial conditions we have $\int_{\omega}u_3^h(t)\,\dd x' =0$ for 
every $t\in (0,T_\eps)$, which implies that there are no translations in the vertical direction. 
Using the trace inequality from Proposition \ref{Poincare} and improved energy estimate (\ref{ineq:energEH}) we have
\begin{align}\label{TranslatioEstimate}
\left|\int_{\omega}u^h_{\alpha}(t)\,\dd x'\right| 
&= \left|\int_{0}^t\frac{\dd}{\dd t}\int_{\omega}u^h_{\alpha}(s)\,\dd x'\dd s\right|
=\left|\int_0^t\int_{\omega}v^{\eps}_{\alpha}(t)\,\dd x'\right|
\leq C\int_0^t\|v_{\alpha}^\eps(t)\|_{L^2(\omega)}
\\
&\leq C\sqrt{\eps}\int_0^t\|\nabla v_{\alpha}^\eps(s)\|_{L^2(\Omega_{\eps})}\dd s 
\leq C\sqrt{\eps}\sqrt{t}\|\nabla v_{\alpha}^\eps(s)\|_{L^2 (0,t;L^2(\Omega_\eps)}\leq C\eps^{2}t\,
\nonumber
\end{align}
for 
every $t\in (0,T_\eps)$. Estimate (\ref{TranslatioEstimate}) shows that for large time scales, which are of particular interest in the
lubrication approximation regime, the horizontal translations can be of order $O(1)$ or bigger.
Moreover, we will see in the subsequent section that these translations do not play any role in 
the derivation of the reduced FSI model (cf.~Section \ref{sec:reduced}), but they do play a role in 
construction of approximate solutions and error analysis (Section \ref{sec:EE}).

\section{Derivation of the reduced FSI model --- proof of Theorem \ref{tm:main}}\label{sec:reduced}

In this section we prove our first main result, Theorem \ref{tm:main}. The proof is divided into
several steps throughout the following subsections. First we employ the scaling ansatz (S1)--(S3),
rescale the energy estimate and obtain uniform estimates on the reference domain. Based
on these estimates we further rescale the unknowns and finally identify the reduced model by
means of weak convergence results. 
\subsection{Uniform estimates on the reference domain} 
The key source of uniform estimates is the energy estimate (\ref{ineq:energEH}). In order to obtain
a nontrivial reduced model we need to rescale the space-time domain and structure data.

\subsubsection{Rescaled energy estimate}
Recall the scaling ansatz (S1)--(S2) and the standard geometric change of variables introduced 
in (\ref{def:UepsH})--(\ref{def:nablaEpsH}).
Let us denote the new time variable with hat and define it
according to $t = \T\hat t$, where $\T>0$ denotes the time scale of the system satisfying (S3).
Functions depending on the new time are then defined by $\hat{\bs w}(\hat t) = \bs w (t)$, 
and its time derivative equals $\pa_t \bs w = \T^{-1}\pa_{\hat t}\hat{\bs w}$. 
Taking all rescalings and change of variables into account, the rescaled energy 
estimate (\ref{ineq:energEH}) on the 
reference domain reads: for a.e.~$t\in(0,T)$
%
\begin{align}\label{ineq:energy_ref}
\frac{\vro_f}{2}\eps\int_{\Omega_-}& |\bv(\eps)(t)|^2 \dd\by
 + \frac{\eta h^\tau}{2}\eps\int_0^t\!\!\int_{\Omega_-}|\nabla_\eps\bv(\eps)|^2 \dd\by\dd s 
 + \frac{\vro_s}{2}h^{-\kappa-2\tau+1}\int_{\Omega_+}|\partial_t\bu(h)(t)|^2 \dd\bz \\
& + h^{-\kappa+1}\int_{\Omega_+}\Big(\mu|\sym\nabla_h\bu(h)(t)|^2 \nonumber
+ \frac{\lambda}{2}|\diver_h\bu(h)(t)|^2\Big)\dd\bz \leq Ch^\tau \eps^{3}\,,
\end{align}
where $T = \T \,T_\eps$ denotes the rescaled time horizon.

\subsubsection{Uniform estimates for the fluid velocity}
The rescaled energy estimate (\ref{ineq:energy_ref}) gives us uniform bound
\begin{equation*}
\frac{\eta h^\tau \eps}{2}\int_0^T\!\!\int_{\Omega_-} |\nabla_\eps\bv(\eps)|^2\dd\by\dd s \leq Ch^\tau\eps^3\,,
\end{equation*} 
which directly from the definition of $\nabla_\eps$ implies
\begin{equation}\label{ineq:pa3v}
\int_0^T\!\!\int_{\Omega_-}|\pa_3\bv(\eps)|^2 \dd\by\dd s \leq C\eps^4\,.
\end{equation}
Using the boundary condition $\left.\bv(\eps)\right|_{\{y_3=-1\}} = 0$, we have the identity
\begin{equation*}
\bv(\eps)(y',y_3,t) = \int_{-1}^{y_3}\pa_3\bv(\eps)(y',\zeta,t)\,\dd\zeta\,,
\end{equation*}
which together with (\ref{ineq:pa3v}) provides
\begin{equation*}
\|\bv(\eps)\|_{L^2(0,T;L^2(\Omega_-;\R^3)} \leq C\eps^2\,.
\end{equation*}
The obtained a priori estimates then imply (on a subsequence as $\eps\downarrow0$):
\begin{align}\label{eq:vel_conv}
\eps^{-2}\bv(\eps) \rightharpoonup \bv\quad \text{and}\quad 
\eps^{-2}\pa_3\bv(\eps) \rightharpoonup \pa_3\bv\quad\text{weakly in }L^2(0,T;L^2(\Omega_-;\R^3))\,.
\end{align}

\subsubsection{Uniform estimates for the pressure}

According to Proposition \ref{prop:exreg} there exists a unique pressure 
$p^\eps\in L^2(0,T;H^1(\Omega_\eps))$ such that the triplet $(\bv^\eps, p^\eps,\bu^h)$ satisfies
the system (\ref{1.eq:stokes})--(\ref{1.eq:elast}) in the $L^2$-sense. Regularity results of 
Proposition \ref{prop:exreg} allow us to weaken the regularity of test functions.
Thus, we multiply (\ref{1.eq:stokes}) and (\ref{1.eq:elast}) by test functions $\bs\phi$
and $\bs\psi$, respectively, where
$(\bs\phi,\bs\psi)\in C_c\left([0,T); \tilde V(\Omega_\eps)\times V_S(\Omega_h)\right)$ such that
 $\bs\phi(t) = \bs\psi(t) \text{ on }\omega \text{ for every }t\in [0,T)$, and
$\tilde V(\Omega_\eps) = \left\{\bv\in H^1(\Omega_\eps;\R^3)\ 
:\ \left.\bv\right|_{\{x_3=-\eps\}}=0\,,\ \bv \text{ is }\omega\text{-periodic} \right\}$.
Integrating with respect to the new (rescaled) time and original space variables we find
\begin{align}
\vro_f\T^{-1}\int_0^T\!\!\!\int_{\Omega_\eps}\pa_t\bv^\eps\cdot\bs\phi\,\dd \bx\dd t \nonumber
 & + 2\eta\int_0^T\!\!\!\int_{\Omega_\eps}\sym\nabla\bv^\eps:\sym\nabla\bs\phi\,\dd \bx\dd t 
 - \int_0^T\!\!\!\int_{\Omega_\eps}p^\eps\diver\bs\phi\,\dd \bx\dd t \label{eq:weak_pressure}\\
+ \vro_s^h\T^{-2}\int_0^T\!\!\!\int_{\Omega_h}\pa_{tt}\bu^h\cdot\bs\psi\,\dd \bx\dd t  & +
\int_0^T\!\!\!\int_{\Omega_h}(2\mu^h\sym\nabla\bu^h:\sym\nabla\bs\psi + \lambda^h\diver\bu^h\diver\bs\psi)\,\dd \bx\dd t \\
= \int_0^T\!\!\!\int_{\Omega_\eps}\bs f^\eps\cdot\bs\phi \,\dd \bx\dd t\,.\nonumber
\end{align}
Unlike in the Stokes equations solely, where the pressure is determined up to a function of time, 
in the case of the FSI problem the pressure is unique. 
This is a consequence of the fact that in the Stokes system the boundary (wall) is assumed to be rigid and 
therefore cannot feel the pressure, while in the present case elastic wall feels the pressure. 
Therefore, we define $\pi^{\eps}(t)=\frac{1}{|\Omega_{\eps}|}\int_{\Omega_{\eps}}p^\eps(\bx,t)\dd\bx$ to be the mean value 
of the pressure at time $t\in(0,T)$. 

Let us first estimate the zero mean value part of the pressure 
$p^\eps-\pi^\eps$ in a classical way.
For an arbitrary $q\in C_c([0,T);L^2_0(\Omega_\eps))$, where 
$L^2_0(\Omega_\eps) = \left\{q\in L^2(\Omega_\eps)\ :\ \int_{\Omega_{\eps}}q\,\dd\bx = 0\right\}$, 
there exists $\bs\phi_q\in C_c([0,T); V(\Omega_\eps))$, 
such that $\diver \bs\phi_q(t) = q(t)$, $\left.\bs\phi_q(t)\right|_\omega = 0$ for all $t\in[0,T)$ and 
$\|\bs\phi_q\|_{L^2(0,T;H^1(\Omega_\eps))}\leq C\eps^{-1}\|q \|_{L^2(0,T;L^2(\Omega_\eps))}$ 
(cf.~\cite[Lemma 9]{MaPa01}).
Taking $(\bs\phi,\bs\psi) = (\bs \phi_q,0)$ 
as test functions in (\ref{eq:weak_pressure}) we have
\begin{align*}
\int_0^T\!\!\!\int_{\Omega_\eps}p^\eps q\,\dd \bx\dd t = 
& \vro_f\T^{-1}\int_0^T\!\!\!\int_{\Omega_\eps}\pa_t\bv^\eps\cdot\bs\phi_q\,\dd \bx\dd t \nonumber
  + 2\eta\int_0^T\!\!\!\int_{\Omega_\eps}\sym\nabla\bv^\eps:\sym\nabla\bs\phi_q\,\dd \bx\dd t 
 \\
& - \int_0^T\!\!\!\int_{\Omega_\eps}\bs f^\eps\cdot\bs\phi_q \,\dd \bx\dd t\,.
\end{align*}
Using the time scaled energy estimates 
(\ref{ineq:energEH}) and (\ref{ineq:energEHH}), assumption
(A1) for the fluid volume force and the Poincar\'e inequality from Proposition \ref{Poincare} 
we conclude
\begin{equation*}
\left|\int_0^T\!\!\!\int_{\Omega_\eps}p^\eps q\,\dd \bx\dd t\right|
\leq C\sqrt{\eps}\|q\|_{L^2(0,T;L^2(\Omega_\eps))}
\end{equation*}
for all $q\in C_c([0,T);L^2_0(\Omega_\eps))$. Employing a density argument, the latter inequality 
implies 
\begin{equation}\label{ineq:pressure0} 
\|p^\eps - \pi^\eps\|_{L^2(0,T;L^2(\Omega_\eps))} \leq C\sqrt{\eps}\,.
\end{equation}

In order to conclude the pressure estimate we still need to estimate the mean value $\pi^{\eps}$.
Let us define test functions by $(\bs\phi,\bs\psi) = \zeta\left ((0,0,x_3+\eps),(0,0,\eps)\right )$
for an arbitrary $\zeta\in C_c([0,T))$. Notice that $\diver\bs\phi(t) = \zeta(t)$ for every $t\in(0,T)$. 
Taking $(\bs\phi,\bs\psi)$ as test functions in \eqref{eq:weak_pressure} we obtain
\begin{align}
\int_0^T\pi^{\eps}\zeta\,\dd t
= \vro_f\T^{-1}\int_0^T\zeta\int_{\Omega_\eps}\left(\frac{x_3}{\eps} + 1\right)\pa_tv_3^\eps\,\dd \bx\dd t
  + \frac{2\eta}{\eps}\int_0^T\zeta\int_{\Omega_\eps}\pa_3 v^\eps_3\,\dd \bx\dd t \label{eq:mean_pressure} \\
  - \int_0^T\zeta\int_{\Omega_\eps} \left(\frac{x_3}{\eps} + 1\right)f_3^\eps \,\dd \bx\dd t
 \,+ \vro_s^h \T^{-2}\int_0^T\zeta\int_{\Omega_h}\,\pa_{tt} u^h_3\,\dd \bx\dd t\,.\nonumber
\end{align}
Let us estimate the right hand side of (\ref{eq:mean_pressure}) using the time scaled energy estimates 
(\ref{ineq:energEH}) and (\ref{ineq:energEHH}):         
\begin{align*}
\left| \int_0^T\pi^{\eps}\zeta\,\dd t \right| \leq C\T^{-1/2}\eps^2\int_0^T|\zeta|\dd t + 
C\eps \|\zeta\|_{L^2(0,T)} + C\T^{-1/2}\eps^{3/2}h^{-\kappa/2+1/2}\int_0^T|\zeta|\dd t\,.
\end{align*}
Under assumption of scaling ansatz (S1) and (S3), and assuming that $\tau = \kappa - 3\gamma - 3\leq - 1$,
the worst term above, $\T^{-1/2}\eps^{3/2}h^{-\kappa/2+1/2}$ is of order less or equal to $O(1)$
(cf.~Section \ref{sec:reduced_model} for the justification of this assumption). 
Therefore, we have
\begin{equation*}
\left| \int_0^T\pi^{\eps}\zeta\,\dd t \right| \leq C\|\zeta\|_{L^2(0,T)}\,,
\end{equation*}
which implies      
\begin{equation}\label{ineq:meanp}
\|\pi^\eps\|_{L^2(0,T;L^2(\Omega_\eps))}\leq C\sqrt \eps\,.
\end{equation}

Combining (\ref{ineq:meanp}) with (\ref{ineq:pressure0}) we find the pressure estimate
\begin{equation*}
\|p^\eps\|_{L^2(0,T;L^2(\Omega_\eps))} \leq C\sqrt{\eps}\,,
\end{equation*}
which further yields the uniform estimate for the pressure $p(\eps)(\by):=p^\eps(\bx)$
defined on the reference domain
\begin{equation}\label{ineq:pressure}
\|p(\eps)\|_{L^2(0,T;L^2(\Omega_+))}\leq C\,. 
\end{equation}
Finally, we conclude that there exists $p\in L^2(0,T;L^2(\Omega_+))$ such that
(on a subsequence as $\eps\downarrow0$) we have
\begin{align}\label{eq:pressure_conv}
p(\eps) \rightharpoonup p\quad \text{weakly in }L^2(0,T;L^2(\Omega_-))\,.
\end{align}

\subsubsection{Uniform estimates for the structure displacement}
The energy estimate (\ref{ineq:energy_ref}) provides an $L^\infty$-$L^2$ estimate of 
the symmetrized scaled gradient,
\begin{equation}\label{ineq:symsg0}
\esssup_{t\in(0,T)}\int_{\Omega_+}|\sym\nabla_h\bu(h)|^2\leq {Ch^{3\gamma-1+\tau + \kappa}}\,.
\end{equation}
With a slight abuse of the notation we introduce the rescaled displacement\\
$\displaystyle\bu(h) = {h^{-(3\gamma-1+\tau + \kappa)/2}}\bu(h)$ 
and (\ref{ineq:symsg0}) transforms into the uniform estimate
\begin{equation}\label{ineq:symsg}
\esssup_{t\in(0,T)}\int_{\Omega_+}\left|\sym\nabla_h\bu(h)\right|^2
\leq C\,.
\end{equation}
In the analysis of structure displacements we rely on the Griso decomposition \cite{Gri05}. 
This is relatively novel ansatz-free approach for the dimension reduction in 
elasticity theory
and with applications in other fields (cf.~\cite{CDG18}). 
For every $h>0$, the structure displacement $\bu(h)$ is, at almost every time instance $t\in(0,T)$, 
decomposed into a sum of so called elementary plate 
displacement and warping as follows (cf.~(\ref{app:Griso_dec}) in Appendix \ref{app:A})
\begin{align}\label{eq:grisodec}
\bu(h)(\bz) = \bs w(h)(z') + 
\bs r(h)(z')\times (z_3-\frac12)\bs e_3 + \tilde{\bu}(h)\,,
\end{align}
where
\begin{align*}
\bs w(h)(t,z') = \int_{0}^1\bu(h)(t,\bz)\dd z_3\,,\quad 
\bs r(h)(t,z') = \frac{3}{h}\int_0^1(z_3-\frac12)\bs e_3\times \bu(h)(t,\bz)\dd z_3\,,
\end{align*}
$\tilde{\bu}(h)\in L^\infty(0,T;H^1(\Omega_+))$ is the warping term, and $\times$ denotes 
the cross product in $\R^3$. Moreover, the following uniform estimate holds
(cf.(\ref{app:Griso_estimate}) in Appendix A)
\begin{align*}
\|\sym\nabla_h \left(\bs w(h)(z') 
+ \bs r(h)(z')\times (z_3-1/2)\bs e_3\right)\|^2_{L^\infty(0,T;L^2(\Omega_+))} + 
\|\nabla_h\tilde{\bu}(h)\|^2_{L^\infty(0,T;L^2(\Omega_+))} \\
+ \frac{1}{h^2}\|\tilde{\bu}(h)\|^2_{L^\infty(0,T;L^2(\Omega_+))} & \leq C\,,
\end{align*}
with $C>0$ independent of $h$ and $\bu(h)$.

According to \cite[Theorem 2.6]{Gri05}, the above
uniform estimate implies the existence of 
a sequence of in-plane translations $\bs a(h) = (a_1(h),a_2(h))\subset (L^\infty(0,T))^2$, as well
as limit displacements $w_1, w_2\in L^\infty(0,T;\Hper^1(\omega))$, $w_3\in L^\infty(0,T;\Hper^2(\omega))$
and $\bar \bu\in L^2(\omega;H^1((0,1);\R^3))$ such that
the following weak-$\star$ convergence results hold:
\begin{align}
w_\alpha(h)- a_\alpha(h) &\overset{\ast}{\rightharpoonup} w_\alpha \quad\text{in }
\ L^\infty(0,T;\Hper^1(\omega))\,,\quad \alpha = 1,2\,,\label{eq:w1_conv} \\
hw_3(h) &\overset{\ast}{\rightharpoonup} w_3 \quad\text{in }\ L^\infty(0,T;\Hper^1(\omega))\,,\label{eq:w3_conv} \\
u_\alpha(h) - a_\alpha(h) &\overset{\ast}{\rightharpoonup} w_\alpha - (z_3-\frac12)\pa_\alpha w_3 
\quad\text{in }\ L^\infty(0,T;H^1(\Omega_+))\,,\quad \alpha = 1,2\,,\label{eq:lsd1}\\
hu_3(h) &\overset{\ast}{\rightharpoonup} w_3  \quad\text{in }\ L^\infty(0,T;H^1(\Omega_+))\,,\label{eq:lsd3}\\
\sym\nabla_h\bu(h)  &\overset{\ast}{\rightharpoonup}  \imath\Big(\sym\nabla' (w_1,w_2)  - (z_3-\frac12)\nabla'^2w_3 \Big)
+ \sym \left(\bs e_3\otimes(\pa_3\bar\bu)\right). \label{eq:lss}
\end{align}
To estimate in-plane translations $a_{\alpha}(h)$ we first use \eqref{TranslatioEstimate} to get:
\begin{equation}\label{TranslationEstimateRescaled}
\left|\int_{\omega}u_{\alpha}(h)(t)\dd x'\right|\leq C\eps^2 h^{\tau - (3\gamma-1+\tau+\kappa)/2}
=C h^{(\tau - \kappa+\gamma+1)/2}\,, \quad t\in [0,T)\,,\quad \alpha = 1,2\,.
\end{equation}
Combining \eqref{TranslationEstimateRescaled} with \eqref{eq:lsd1} we get
\begin{equation}\label{TranslatioLimitestimates}
\|a_{\alpha}(h)\|_{L^{\infty}(0,T)}\leq C h^{(\gamma+\tau+1-\kappa)/2}\,,
\end{equation}
and therefore $h^{(\kappa-\gamma-\tau-1)/2}a_\alpha(h) \overset{\ast}{\rightharpoonup} a_\alpha$ in
$L^\infty(0,T)$.

Employing the higher-order energy estimate (\ref{ineq:energEHH}) on the reference domain and in the
rescaled time we find convergence results for the respective 
time derivatives analogous to (\ref{eq:w1_conv})--(\ref{eq:lss}). Moreover, it holds
\begin{equation}\label{ineq:transl_derivative}
\|\partial_t a_{\alpha}(h)\|_{L^{\infty}(0,T)}\leq Ch^{(\tau - \kappa + \gamma + 1)/2}\,,
\end{equation}
which implies 
$h^{(\kappa-\gamma-\tau-1)/2}\pa_ta_\alpha(h) \overset{\ast}{\rightharpoonup} \pa_t a_\alpha$ in
$L^\infty(0,T)$.

\subsection{Identification of the reduced model}\label{sec:reduced_model}
Taking all rescalings into account,
the rescaled variational equation (\ref{eq:weak_pressure}) on the reference domain reads
\begin{align*}
\nonumber
-\vro_f h^{-\tau} \eps^3 \int_0^T\!\!\int_{\Omega_-}\bv(\eps)\cdot\pa_t\bs\phi\,\dd \by \dd t 
+ 2\eta\eps^3\int_0^T\!\!\int_{\Omega_-}\sym\nabla_\eps\bv(\eps):\sym\nabla_\eps\bs\phi\,\dd \by\dd t \\
- \eps\int_0^T\!\!\int_{\Omega_-}p(\eps)\diver_\eps\bs\phi\,\dd \by\dd t + 
\vro_sh^{\delta - 2\tau}\int_0^T\!\! \int_{\Omega_+}\bu(h)\cdot\pa_{tt}\bs\psi \,\dd \bz\dd t \\  + 
h^\delta\int_0^T\!\!\int_{\Omega_+}(2\mu\sym\nabla_h\bu(h):\sym\nabla_h\bs\psi 
+ \lambda\diver_h\bu(h)\diver_h\bs\psi)\,\dd \bz\dd t \\
 = \eps\int_0^T\!\!\int_{\Omega_-}\bs f(\eps)\cdot\bs\phi \,\dd \by\dd t\,, 
\end{align*}
for all $(\bs\phi,\bs\psi)\in C_c^2\left([0,T); V(\Omega_-)\times V_S(\Omega_+)\right)$ such that 
 $\bs\phi(t) = \bs\psi(t) \text{ on }\omega \text{ for all }t\in [0,T)$,
 and where $\delta = - \kappa + (3\gamma-1+\tau + \kappa)/2 + 1 $.
In order to obtain a nontrivial coupled reduced model on a limit as $h\downarrow0$ 
we need to adjust $\delta=-1$. This condition is due to the linear theory of plates 
(cf.~\cite[Section 1.10]{Cia97}). Namely, the fluid 
pressure which is here $O(1)$ acts as a normal force on the structure and therefore 
has to balance the structure stress terms in the right way.
Hence, we find the choice of the right time scale to be $\T = h^\tau$ with  
\begin{equation}\label{eq:rel_tau}
\tau = \kappa - 3\gamma - 3\,.
\end{equation} 
The above weak formulation then becomes
\begin{align}
\nonumber
-\vro_f h^{6\gamma - \kappa + 3} \int_0^T\!\!\int_{\Omega_-}\bv(\eps)\cdot\pa_t\bs\phi\,\dd \by \dd t
+ 2\eta\eps^3\int_0^T\!\!\int_{\Omega_-}\sym\nabla_\eps\bv(\eps):\sym\nabla_\eps\bs\phi\,\dd \by\dd t \\
- \eps\int_0^T\!\!\int_{\Omega_-}p(\eps)\diver_\eps\bs\phi\,\dd \by\dd t  +
\vro_sh^{6\gamma - 2\kappa + 5}\int_0^T\!\! \int_{\Omega_+}\bu(h)\cdot\pa_{tt}\bs\psi \,\dd \bz\dd t \label{eq:weak_rescaled} \\  +
h^{-1}\int_0^T\!\!\int_{\Omega_+}(2\mu\sym\nabla_h\bu(h):\sym\nabla_h\bs\psi
+ \lambda\diver_h\bu(h)\diver_h\bs\psi)\,\dd \bz\dd t \nonumber\\
 = \eps\int_0^T\!\!\int_{\Omega_-}\bs f(\eps)\cdot\bs\phi \,\dd \by\dd t\,.\nonumber 
\end{align}

Expanding \ref{eq:weak_rescaled} with test functions of the form $\bs\phi = (\phi_1,\phi_2,0)$ and
$\bs \psi = (\psi_1,\psi_2,0)$, multiplying the equation with $h^2$ and employing the weak*-convergence 
results for the structure (\ref{eq:lsd1})--(\ref{eq:lss}), we find 
\begin{align*}
\mu\int_0^T\!\!\int_{\Omega_+}(\pa_3\bar u_1\pa_3\psi_1 + \pa_3\bar u_2\pa_3\psi_2)\,\dd \bz\dd t = 0\,.
\end{align*}
Taking sequences $(\psi_{1,n})$ and $(\psi_{2,n})$ which approximate $\bar u_1$ and $\bar u_2$, 
respectively, in the sense of $L^2$-convergence, we conclude $\pa_3\bar u_1 = \pa_3\bar u_2 = 0$.
Similarly, taking test functions $\bs\phi = (0, 0, \phi_3)$ and
$\bs \psi = (0, 0, \psi_3)$, we obtain
\begin{align*}
\mu\int_0^T\!\!\int_{\Omega_+}\big((2\mu + \lambda)\pa_3\bar u_3 + \lambda(\pa_1 w_1 + \pa_2w_2 
- (z_3 - \frac12)\Delta'w_3) \big)\pa_3\psi_3 \,\dd \bz\dd t = 0\,,
\end{align*}
from which we conclude
\begin{align*}
\pa_3\bar u_3 = -\frac{\lambda}{2\mu + \lambda}\big(\pa_1 w_1 + \pa_2w_2 
- (z_3 - \frac12)\Delta' w_3\big)\,.
\end{align*}
Previous calculations are motivated with those from the proof of Theorem 1.4-1 from \cite{Cia97}. 
Now we have  complete information on the limit of the scaled strain (\ref{eq:lss}) 
given in terms of the limit displacements $(w_1,w_2,w_3)$.

Next, we will take test functions to imitate the shape of the limit of scaled displacements 
(\ref{eq:lsd1})--(\ref{eq:lsd3}), i.e.~we take $\bs\psi = (h\psi_1,h\psi_2,\psi_3)$ satisfying 
$\pa_1\psi_3 + \pa_3\psi_1 = \pa_2\psi_3 + \pa_3\psi_2 = \pa_3\psi_3 = 0$,
while for the fluid part we accordingly take (in order to satisfy the interface conditions) 
$\bs \phi = (h\phi_1,h\phi_2,\phi_3)$.
With this choice of test functions, under assumption $\tau \leq -1$ (i.e.~$\kappa\leq 3\gamma + 2$), 
the weak limit form of (\ref{eq:weak_rescaled}) 
(on a subsequence as $h\downarrow0$) reads
\begin{align}
-\int_0^T\!\!\int_{\Omega_-}p\pa_3\phi_3\,\dd \by\dd t + \nonumber
\chi_\tau\vro_s\int_0^T\!\! \int_{\omega}w_3\pa_{tt}\psi_3 \,\dd z'\dd t \\ \label{eq:limit_fsi}
+ \int_0^T\!\!\int_{\Omega_+}\Big(2\mu\big(\sym\nabla' (w_1,w_2) - (z_3-\frac12)\nabla'^2w_3\big):\sym\nabla'(\psi_1,\psi_2)\\
+ \frac{2\mu\lambda}{2\mu + \lambda}\diver\big((w_1,w_2) - (z_3-\frac12)\nabla' w_3 \big)\diver(\psi_1,\psi_2)\Big)\,\dd \bz\dd t & = 0\,,\nonumber
\end{align}
where $\chi_\tau = 1$ for $\tau = -1$ and $\chi_\tau = 0$ for $\tau < -1$. Notice that for
$\tau = -1$, i.e.~$\kappa = 3\gamma + 2$, we have $\delta - 2\tau = 1$, and the inertial term 
of the vertical displacement of the structure survives in the limit. 

The obtained limit model (\ref{eq:limit_fsi}) is a linear plate model (cf.~\cite{Cia97}) coupled with the limit 
pressure from the fluid part, which acts as a normal force on the interface $\omega$
of the structure (cf.~equation (\ref{eq:limit_fsi2}) below).
Let us consider the pressure term more in detail. Taking test function 
$\bs\phi\in C_c^1([0,T);C_c^\infty(\Omega_-;\R^3))$, i.e.~smooth and with compact support in space, and 
$\bs\psi = 0$ in (\ref{eq:weak_rescaled}) we find
\begin{align*}
\int_0^T\!\!\int_{\Omega_-}p\pa_3\phi_3\,\dd \by\dd t = 0\,,
\end{align*}
which implies $\pa_3 p=0$ in the sense of distributions. As a consequence of this we have that 
$p$ is independent of the vertical variable $z_3$, and therefore $p$ (although $L^2$-function)
has the trace on $\omega$. Since $\phi_3 = \psi_3$ on $\omega\times(0,T)$, after integrating by parts 
in the pressure term, the limit form (\ref{eq:limit_fsi}) then becomes
\begin{align}
-\int_0^T\!\!\int_{\omega}p\psi_3\,\dd z'\dd t + \nonumber
\chi_\tau\vro_s\int_0^T\!\! \int_{\omega}w_3\pa_{tt}\psi_3 \,\dd z'\dd t \\ \label{eq:limit_fsi2} 
+ \int_0^T\!\!\int_{\Omega_+}\Big(2\mu\big(\sym\nabla' (w_1,w_2) - (z_3-\frac12)\nabla'^2w_3\big):\sym\nabla'(\psi_1,\psi_2)\\
+ \frac{2\mu\lambda}{2\mu + \lambda}\diver\big((w_1,w_2) - (z_3-\frac12)\nabla' w_3 \big)\diver(\psi_1,\psi_2)\Big)\,\dd \bz\dd t & = 0\,.\nonumber
\end{align}

Recall that structure test functions in (\ref{eq:limit_fsi2}) satisfy 
$\pa_1\psi_3 + \pa_3\psi_1 = \pa_2\psi_3 + \pa_3\psi_2 = \pa_3\psi_3 = 0$. 
According to \cite[Theorem 1.4-1 (c)]{Cia97}, these conditions are equivalent with 
the following representation of test functions:
\begin{align*}
\psi_\alpha = \zeta_\alpha - (z_3 - \frac12)\pa_\alpha\zeta_3\,,\quad\text{and}\quad \psi_3 = \zeta_3\,,
\end{align*}
for some $\zeta_\alpha\in C_c^1([0,T);\Hper^1(\omega))$, $\alpha=1,2$, 
and $\zeta_3\in C_c^1([0,T);\Hper^2(\omega))$.
Next, we resolve (\ref{eq:limit_fsi2}) into equivalent formulation, which decouples horizontal and 
vertical displacements. 

First, choosing the test function  
$\bs \psi = (-(z_3-\frac12)\pa_1\zeta_3, -(z_3-\frac12)\pa_2\zeta_3, \zeta_3)$, for arbitrary
$\zeta_3\in C_c^2([0,T);\Hper^2(\omega))$,
after explicit calculations of integrals we find
\begin{align}\label{eq:w3_1}
-\int_0^T\!\!\int_{\omega}p\zeta_3\,\dd z'\dd t +
\chi_\tau\vro_s\int_0^T\!\! \int_{\omega}w_3\pa_{tt}\zeta_3 \,\dd z'\dd t &~\\ +
\int_0^T\!\!\int_{\omega}\left(\frac{4\mu}{3}\nabla'^2 w_3:\nabla'^2\zeta_3
+ \frac{4\mu\lambda}{3(2\mu + \lambda)}\Delta' w_3\Delta' \zeta_3\right) \,\dd z'\dd t \nonumber
&= 0\,.
\end{align}
Secondly, taking the test function $\bs\psi = (\zeta_1,\zeta_2,0)$, for arbitrary
$\zeta_\alpha\in C_c^1([0,T);\Hper^1(\omega))$,
we obtain the variational equation for horizontal displacements only,
\begin{align}\label{eq:w12_1}
\int_0^T\!\!\int_{\omega}\left(4\mu\sym\nabla' (w_1,w_2) :\nabla'(\zeta_1,\zeta_2)
+ \frac{4\mu\lambda}{2\mu + \lambda}\diver'\left(w_1,w_2\right)\diver'(\zeta_1,\zeta_2)\right)\dd z'\dd t
= 0\,.
\end{align}
Equation (\ref{eq:w12_1}) implies that horizontal displacements $(w_1,w_2)$ are spatially constant
functions, and as such they will not affect the reduced model. 
Moreover, they are dominated by potentially large horizontal translations, hence we omit them
in further analysis. Thus, the limit system
(\ref{eq:limit_fsi2}) is now essen\-ti\-a\-lly described with (\ref{eq:w3_1}), which relates the limit 
fluid pressure $p$ with the limit vertical displacement of the structure $w_3$. 

In order to close 
the limit model, we need to further explore on the fluid part. 
First, we analyze the divergence free
condition on the reference domain. Multiplying $\diver_\eps\bv(\eps) = 0$ by a test function
$\varphi\in C_c^1([0,T);\Hper^1(\omega))$, integrating over space and time, integrating by parts and
employing the rescaled kinematic condition 
$\eps^2\bv(\eps) = h^{(3\gamma-1+\tau + \kappa)/2 - \tau}\pa_t\bu(h)$, which with relation 
(\ref{eq:rel_tau}) and (S1) becomes $\eps^{-1}\bv(\eps) = h\pa_t\bu(h)$ a.e.~on $\omega\times(0,T)$,
we have
\begin{align}\label{eq:divh}
-\int_0^T\!\!\int_{\Omega_-}\left(v_1(\eps)\pa_1\varphi + v_2(\eps)\pa_2\varphi \right)\dd \by\dd t - 
\int_0^T\!\!\int_{\omega}hu_3(h)\pa_t\varphi\,\dd y'\dd t = 0\,.
\end{align}
Utilizing convergence results (\ref{eq:vel_conv}) and (\ref{eq:lsd3}) in (\ref{eq:divh}),
we find (on a subsequence as $\eps\downarrow0$)
\begin{align}\label{eq:div_limit}
-\int_0^T\!\!\int_{\Omega_-}\left(v_1\pa_1\varphi + v_2\pa_2\varphi \right)\dd \by\dd t - 
\int_0^T\!\!\int_{\omega}w_3\pa_t\varphi\,\dd y'\dd t = 0\,,
\end{align}
which relates the limit vertical displacement of the structure with limit horizontal fluid velocities.

The relation between horizontal fluid velocities $(v_1,v_2)$ and pressure $p$ is obtained from
(\ref{eq:weak_rescaled}) as follows. Take test functions $\bs\phi = (\phi_1/\eps,\phi_2/\eps,0)$ with 
$\phi_\alpha\in C_c^1([0,T);C_c^\infty(\Omega_-))$ and
$\bs\psi = 0$, then convergence results (\ref{eq:vel_conv}) and (\ref{eq:pressure_conv}) yield
\begin{align}\label{eq:pv}
\eta\int_0^T\!\!\int_{\Omega_-}\left(\pa_3v_1\pa_3\phi_1 + \pa_3v_2\pa_3\phi_2\right)\dd\by\dd t
- \int_0^T\!\!\int_{\Omega_-}& \left(p\pa_1\phi_1 + p\pa_2\phi_2 \right)\dd\by\dd t \\
 &= \int_0^T\!\!\int_{\Omega_-}\left(f_1\phi_1 + f_2\phi_2\right)\dd\by\dd t\,, \nonumber
\end{align}
and the reduced model composed of (\ref{eq:w3_1}), (\ref{eq:div_limit}) and (\ref{eq:pv}) 
is now closed.

Before exploring the limit model more in detail, let us conclude that $v_3 = 0$. Namely,
for an arbitrary $\varphi\in C_c^1([0,T);H^1(\Omega_-))$, integrating
 the divergence-free condition we calculate 
\begin{align*}
\lim_{\eps\downarrow0}\int_0^T\!\!\int_{\Omega_-}\pa_3v_3(\eps)\varphi\,\dd\by\dd t
= \lim_{\eps\downarrow0}\left(\eps\int_0^T\!\!\int_{\Omega_-}\left(v_1(\eps)\pa_1\varphi 
+ v_2(\eps)\pa_2\varphi\right)\dd\by\dd t \right) = 0\,,
\end{align*}
which implies $\pa_3v_3=0$, and therefore $v_3=0$, due to the no-slip boundary condition.

\subsection{A single equation}\label{sec:thinfilmeq}
Since $p$ is independent of the vertical variable $y_3$, equation (\ref{eq:pv})
can be solved for $v_\alpha$ explicitly in terms of $y_3$ and $p$. 
Let us first resolve the boundary conditions for $v_\alpha$ in the vertical direction.
The bottom condition is inherited from the original no-slip condition, 
i.e.~$v_\alpha(\cdot,-1,\cdot) = 0$,
while for the interface condition we derive $v_\alpha(\cdot,0,\cdot) = \pa_ta_\alpha$, $\alpha = 1,2$,
where $\pa_ta_\alpha$ are translational limit velocities of the structure defined by 
(\ref{ineq:transl_derivative}). 
Recall the rescaled
kinematic condition $v_\alpha(\eps) = \eps h\pa_tu_\alpha(h)$ on $\omega\times(0,T)$. Multiplying this
with a test function $\varphi\in C_c^1([0,T); H^1(\omega))$ and using convergence results
(\ref{eq:lsd1}) and (\ref{TranslatioLimitestimates}) we have
\begin{equation*}
\lim_{\eps\downarrow0}\int_0^T\!\!\int_{\omega}v_\alpha(\eps) \varphi\,\dd y'\dd t =
- \lim_{h\downarrow0}h^{\gamma+1}\int_0^T\!\!\int_{\omega}u_\alpha(h)\,\pa_t\varphi \,\dd z'\dd t
= {\int_0^T\!\!\int_\omega\partial_t a_{\alpha}\,\varphi}\,\dd z'\dd t \,.
\end{equation*}
Since (\ref{eq:vel_conv}) implies $v_\alpha(\eps)\rightharpoonup v_\alpha$ weakly
in $L^2(0,T;L^2(\omega))$, we conclude that $v_\alpha = {\partial_t a_{\alpha}}$
a.e.~on $\omega\times(0,T)$.
Explicit solution of $v_\alpha$ ($\alpha=1,2$) from (\ref{eq:pv}) is then given by
\begin{equation}\label{eq:v_alpha}
v_\alpha(\by,t) = \frac{1}{2\eta}y_3(y_3+1)\pa_\alpha p(y',t) + F_\alpha(\by,t) + (1+y_3)\pa_ta_\alpha\,,
\quad (\by,t)\in\Omega_-\times(0,T)\,,
\end{equation} 
where $\displaystyle F_\alpha(\cdot,y_3,\cdot) 
= \frac{y_3+1}{\eta}\int_{-1}^{0} \zeta_3 f_\alpha(\cdot,\zeta_3,\cdot)\,\dd \zeta_3 + 
\frac{1}{\eta}\int_{-1}^{y_3}(y_3-\zeta_3) f_\alpha(\cdot,\zeta_3,\cdot)\,\dd \zeta_3$. 
From equation (\ref{eq:div_limit}) we have 
\begin{equation}\label{eq:reynolds1} 
\int_0^T\!\!\int_{\omega}\left(\pa_1\int_{-1}^0 v_1\,\dd y_3 + \pa_2\int_{-1}^0 v_2\,\dd y_3
+ \pa_t w_3\right)\varphi\,\dd y'\dd t = 0\,.
\end{equation} 
Replacing $v_\alpha$ with (\ref{eq:v_alpha}) it follows a Reynolds type equation
\begin{equation}\label{eq:reynolds2}
\int_0^T\!\!\int_{\omega}\left(-\frac{1}{12\eta}\Delta'p - F
+ \pa_t w_3\right)\varphi\,\dd y'\dd t = 0\,,
\end{equation}
where $\displaystyle F(y',t) = -\int_{-1}^0 \left( \pa_1F_1 + \pa_2 F_2\right)\dd y_3$.
Considering equation (\ref{eq:w3_1}) in the sense of distributions, i.e.
\begin{align}\label{eq:pw3}
p &= \chi_\tau\vro_s\pa_{tt}w_3 + 
\frac{8\mu(\mu + \lambda)}{3(2\mu + \lambda)}(\Delta')^2w_3\,,
\end{align}
where $(\Delta')^2$ denotes the bi-Laplacian in horizontal variables, we finally
obtain the reduced model in terms of the vertical displacement only
\begin{equation}\label{eq:w3_evol}
\pa_t w_3 - \chi_\tau\frac{\vro_s}{12\eta}\Delta'\pa_{tt}w_3 
- \frac{2\mu(\mu + \lambda)}{9\eta(2\mu + \lambda)}(\Delta')^3w_3 = F \,. 
\end{equation}
This is an evolution equation for $w_3$ of order six in spatial derivatives. 
The term with mixed space and time 
derivatives is present only for $\tau=-1$, and in the context of beam models it is 
called a rotational
inertia (cf.~\cite[Section 1.14]{Cia97}). Equation (\ref{eq:w3_evol}) is 
accompanied by trivial initial data $w_3(0) = 0$ and periodic boundary conditions.
Knowing $w_3$, the pressure and horizontal velocities of the fluid are then calculated according to
(\ref{eq:pw3}) and (\ref{eq:v_alpha}), respectively. This finishes the proof of Theorem \ref{tm:main}.

\section{Error estimates --- proof of Theorem \ref{tm:EE}}\label{sec:EE}  
This section is devoted to the proof of our second main result --- Theorem \ref{tm:EE}, 
which provides the error estimates for approximation of solutions to the original FSI problem 
(\ref{1.eq:stokes})--(\ref{def:InitialCond}) by
approximate solutions constructed from the reduced model (\ref{eq:w3_evol}). 
In the subsequent analysis we will assume additional regularity of 
solutions $(\bv^\eps, p^\eps, \bu^h)$
 to the original problem, together with
 sufficient regularity of solutions $w_3$ to the reduced problem (\ref{eq:w3_evol}), as well as 
 regularity of external forces.  
In the sequel we work on the original thin domain $\Omega_{\eps,h}$, but in the rescaled time variable
with the scaling parameter $\tau < -1$.

\subsection{Construction of approximate solutions and error equation}

Recall the limit model (\ref{eq:w3_evol}) in terms of the scaled vertical displacement $w_3$ (for $\tau < -1$):
\begin{align}\label{eq:limit_model}  
\pa_t w_3 
- \frac{2\mu(\mu + \lambda)}{9\eta(2\mu + \lambda)}(\Delta')^3w_3 &= F \,
\quad \text{in }\ \omega\times(0,T)\,, \\ \nonumber
w_3(0) &= 0\,.
\end{align}
This is a linear parabolic partial differential equation with periodic boundary conditions.
The classical theory of linear partial differential equations provides the well-posedness and smoothness 
of the solution (see e.g. \cite[Chapters 3 and 4]{LionsMagenes}).
Based on (\ref{eq:limit_model}) we reconstruct the limit fluid pressure and horizontal velocities
according to:
\begin{align}\label{eq:limit_p}
p &= 
\frac{8\mu(\mu + \lambda)}{3(2\mu + \lambda)}(\Delta')^2w_3\,,\\
v_\alpha &= \frac{1}{2\eta}y_3(y_3 + 1)\pa_\alpha p + F_\alpha  + (y_3+1)\pa_ta_\alpha
\,,\quad \alpha=1,2\,,
\label{eq:limit_v}
\end{align} 
where $F_\alpha$ is defined like in (\ref{eq:v_alpha}). Limit structure velocities $\pa_t a_\alpha$
will be specified by an 
additional interface condition $\displaystyle \int_\omega \pa_3 v_\alpha \, \dd z' = 0$ for a.e.~$t\in(0,T)$, 
which can be formally seen as a weakened limit stress balance condition and it will be justified by the
convergence result of Theorem \ref{tm:EE}. Using the periodic boundary conditions of the pressure, the
interface condition implies
\begin{equation}\label{def:pata}
\pa_ta_\alpha (t) = -\int_\omega \pa_3F_\alpha(y',0,t)\dd y'\,,\quad \alpha = 1,2\,.
\end{equation}
Let us first define the approximate pressure by
\begin{align*}
\paa^\eps(\bx,t) = p(x',t)\,,\quad (\bx,t)\in\Omega_\eps\times(0,T)\,,
\end{align*}
where $p$ is given by (\ref{eq:limit_p}).
An approximate fluid velocity is constructed as (cf.~\cite{DuMP00})
\begin{align}\label{eq:Approx}
{\bs\va}^\eps(\bx,t)=\eps^2\Big(v_1(x',\frac{x_3}{\eps},t),v_2(x',\frac{x_3}{\eps},t),\va_3^\eps(\bx,t)\Big)\,,
\quad (\bx,t)\in\Omega_\eps\times(0,T)\,,
\end{align}
where $v_1$, $v_2$ are given by (\ref{eq:limit_v}) and 
\begin{equation*}
\va_3^\eps(\bx,t)=-\eps^3\int_{-1}^{x_3/\eps}(\pa_1v_1
+ \pa_2v_2)(x',\xi,t)\,\dd \xi\,.
\end{equation*}
Notice that $\diver{\bs\va}^\eps = 0$ and therefore 
${\bs\va}^\eps \in L^2(0,T;V_F(\Omega_\eps))$. Furthermore, $\paa^\eps$ and ${\bs\va}^\eps$
solve the modified Stokes system
\begin{equation}\label{eq:modifStokes}
\varrho_f\T^{-1}\partial_t \bs\va^\eps - \diver\sigma_f(\bs\va^\eps,\paa^\eps) 
= \bs f^\eps - f_3^\eps\bs e_3 + \bs \res^\eps_f\,,
\end{equation}
where the residual term $\bs \res^\eps_f$ is given by
\begin{equation}\label{eq:res_fluid}
\bs \res^\eps_f = \varrho_f\T^{-1}\partial_t \bs\va^\eps 
- \eta\Delta'{\bs\va}^\eps
-\eta\partial_{33} \va_3^\eps\,\bs e_3\,.
\end{equation}
From the definition of the fluid residual $\bs\res_f^\eps$ we immediately have 
\begin{align}\label{ineq:Residual}
\|\bs \res_f^\eps\|_{L^2(0,T;L^2(\Omega_{\eps}))}\leq C\eps^{3/2}\,.
\end{align}

Multiplying equation (\ref{eq:modifStokes}) by a test function $\bs\phi\in C^1_c([0,T);V_F(\Omega_\eps))$,
 and then integrating over $\Omega_\eps\times(0,T)$, we find
\begin{align}\label{eq:modifStokes_weak}
-\vro_f\int_0^T\!\! \int_{\Omega_\eps}\bs\va^\eps\cdot\pa_t\bs\phi\,\dd \bx\dd t
+& 2\eta\T\int_0^T\!\!\int_{\Omega_\eps}\sym\nabla\bs\va^\eps:\sym\nabla\bs\phi\,\dd \bx\dd t \\
-\T\int_0^T\!\!\int_{\omega}\sigma_f(\bs\va^\eps,\paa^\eps)\bs\phi\cdot \bs e_3\,\dd x'\dd t 
&= \T\int_0^T\!\!\int_{\Omega_\eps}(f_1^\eps\phi_1 + f_2^\eps\phi_2) \,\dd \bx\dd t 
+ \T\int_0^T\!\!\int_{\Omega_\eps}\bs\res_f^\eps\cdot\bs\phi \,\dd \bx\dd t \,.\nonumber
\end{align}
Expanding the boundary term we get
\begin{align*}
\T\int_0^T\!\!\int_{\omega}\sigma_f & (\bs\va^\eps,\paa^\eps) \bs\phi\cdot \bs e_3\,\dd x'\dd t \\
&= \T\int_0^T\!\!\int_{\omega}\Big(-\eps^3\eta\int_{-1}^0(\pa_{11}v_1 + \pa_{12}v_2)\phi_1
-\eps^3\eta\int_{-1}^0(\pa_{21}v_1 + \pa_{22}v_2)\phi_2 \\
&\qquad\qquad  + \eps\eta\pa_3v_1\phi_1 + \eps\eta\pa_3v_2\phi_2 
- 2\eps^2\eta(\pa_1v_1 + \pa_2v_2)\phi_3 - p\phi_3 \Big)\,\dd x'\dd t \\
&= \T\int_0^T\!\!\int_{\omega}\bs\res_b^\eps\cdot\bs\phi \,\dd x'\dd t 
- \T\int_0^T\!\!\int_{\omega}p\phi_3 \,\dd x'\dd t\,, 
\end{align*}
where we defined $\bs\res_b^\eps$ as the boundary residual term given by
\begin{align*}
\bs\res_b^\eps = \eta\left(\eps\pa_3v_1 -
\eps^3\int_{-1}^0(\pa_{11}v_1 + \pa_{12}v_2), \eps\pa_3v_2 -
\eps^3\int_{-1}^0(\pa_{21}v_1 + \pa_{22}v_2), 
- 2\eps^2(\pa_1v_1 + \pa_2v_2)\right).   
\end{align*}


Next, we define the approximate displacement by  
\begin{equation}\label{def:ua_h}
\bs\ua^h(\bx,t) = h^{\kappa-3}\left(h^{-\gamma}a_1 - \Big(x_3 
- \frac{h}{2}\Big)\pa_1w_3(x',t), h^{-\gamma}a_2 - \Big(x_3 - \frac{h}{2}\Big)\pa_2w_3(x',t), 
w_3(x',t) \right)\,,   
\end{equation}
for all $(\bx,t)\in\Omega_h\times(0,T)$, 
where $w_3$ is the solution of (\ref{eq:limit_model}), and $a_\alpha$ are horizontal time-dependent 
translations calculated by $\displaystyle a_\alpha(t) = \int_0^t\pa_ta_\alpha \dd s$, $\alpha = 1,2$,
with $\pa_ta_\alpha$ given by (\ref{def:pata}).
According to the limit form (\ref{eq:limit_fsi2}), the pressure $p$ and the approximate displacement 
$\bs\ua^h$ are related through
\begin{align*}
&-\int_0^T\!\!\int_{\omega}p\phi_3\,\dd x'\dd t = \nonumber\\
&- \int_0^T\!\!\int_{\Omega_h}\Big(2\mu^h\sym\nabla\bs\ua^h :\imath(\sym\nabla'(h\psi_1,h\psi_2))
+ \frac{2\mu^h\lambda^h}{2\mu^h + \lambda^h}
\diver\bs\ua^h \diver(h\psi_1,h\psi_2)\Big)\,\dd \bx\dd t  \,
\end{align*}
for all test functions $\bs\psi$ 
satisfying $\pa_1\psi_3 + \pa_3\psi_1 = \pa_2\psi_3 + \pa_3\psi_2 = \pa_3\psi_3 = 0$
and $\psi_3 = \phi_3$ on $\omega$.
Furthermore, since $\sym\nabla\bs\ua^h$ has only $2\times2$ nontrivial submatrix, the latter
identity can be written as
\begin{align*}
\int_0^T\!\!\int_{\omega}p\phi_3\,\dd x'\dd t = & 
  \int_0^T\!\!\int_{\Omega_h}\left(2\mu^h\sym\nabla\bs\ua^h :\sym\nabla \bs\psi 
+ \lambda^h\diver\bs\ua^h \diver \bs\psi 
\right)\,\dd \bx\dd t  \\
& - \int_0^T\!\!\int_{\Omega_h}\Big(\frac{(\lambda^h)^2}{2\mu^h + \lambda^h}\diver\bs\ua^h 
\diver(h\psi_1,h\psi_2) + \lambda^h\diver\bs\ua^h \pa_3\psi_3\Big)\,\dd \bx\dd t\,,
\end{align*}
with a test function $\bs\psi = (h\psi_1,h\psi_2,\psi_3)$ which now satisfies only $\psi_3 = \phi_3$ on $\omega$.
Going back to (\ref{eq:modifStokes_weak}) and taking $\bs\psi = (\psi_1,\psi_3,\psi_3)$ in 
further calculations, 
we find the weak form of approximate solutions to be of the same type
as the original weak formulation (\ref{eq:weak_fsi}) with additional residual terms:
\begin{align}
-\vro_f\int_0^T\!\! \int_{\Omega_\eps}\bs\va^\eps\cdot\pa_t\bs\phi\,\dd \bx\dd t \nonumber
+ 2\eta\T\int_0^T\!\!\int_{\Omega_\eps}\sym\nabla\bs\va^\eps:\sym\nabla\bs\phi\,\dd \bx\dd t\\
-\vro_s^h\T^{-1}\int_0^T\!\!\int_{\Omega_h}\pa_t\bs\ua^h\cdot\pa_t\bs\psi\,\dd\bx\dd t 
+ \T\int_0^T\!\!\int_{\Omega_h}\left(2\mu^h\sym\nabla\bs\ua^h :\sym\nabla\bs\psi
+ \lambda^h\diver\bs\ua^h \diver\bs\psi\right)\,\dd \bx\dd t \label{eq:weak_fsi_app}  \\
= \T\int_0^T\!\!\int_{\Omega_\eps}(f_1^\eps\phi_1 + f_2^\eps\phi_2) \,\dd \bx\dd t 
+ \T\int_0^T\!\!\int_{\Omega_\eps}\bs\res_f^\eps\cdot\bs\phi \,\dd \bx\dd t \nonumber
+ \T\int_0^T\!\!\int_{\omega}\bs\res_b^\eps\cdot\bs\phi \,\dd x'\dd t 
+ \langle \bs\res_s^h,\bs\psi\rangle\,,
\end{align}
where $\langle \bs\res_s^h,\bs\psi\rangle$ denotes the structure residual term $\bs\res_s^h$ acting
on a test function $\bs\psi$ as
\begin{align*}
\langle \bs\res_s^h,\bs\psi\rangle &= 
-\vro_s^h\T^{-1}\int_0^T\!\!\int_{\Omega_h}\pa_t\bs\ua^h\cdot\pa_t\bs\psi\,\dd\bx\dd t\\
&\quad + \T\int_0^T\!\!\int_{\Omega_h}\Big(\frac{(\lambda^h)^2}{2\mu^h + \lambda^h}\diver\bs\ua^h 
\diver(\psi_1,\psi_2) + \lambda^h\diver\bs\ua^h \pa_3\psi_3\Big)\,\dd \bx\dd t\,.
\end{align*}

Let us define the fluid error $\bs e_f^\eps:= \bv^\eps - \bs\va^\eps$ 
and the structure error $\bs e_s^h:= \bu^h - \bs\ua^h$.
Subtracting (\ref{eq:weak_fsi_app}) from the original problem (\ref{eq:weak_fsi}), 
in the rescaled time, we find the variational equation for the errors:
\begin{align}
-\vro_f\int_0^T\!\! \int_{\Omega_\eps}\bs e^\eps_f\cdot\pa_t\bs\phi\,\dd \bx\dd t \nonumber
+ 2\eta\T\int_0^T\!\!\int_{\Omega_\eps}\sym\nabla\bs e_f^\eps:\sym\nabla\bs\phi\,\dd \bx\dd t\\
-\vro_s^h\T^{-1}\int_0^T\!\!\int_{\Omega_h}\pa_t\bs e_s^h\cdot\pa_t\bs\psi\,\dd\bx\dd t 
+ \T\int_0^T\!\!\int_{\Omega_h}\left(2\mu^h\sym\nabla\bs e_s^h :\sym\nabla\bs\psi
+ \lambda^h\diver\bs e_s^h \diver\bs\psi\right)\,\dd \bx\dd t \label{eq:weak_fsi_error}   \\
= \T\int_0^T\!\!\int_{\Omega_\eps}f_3^\eps\phi_3 \,\dd \bx\dd t 
- \T\int_0^T\!\!\int_{\Omega_\eps}\bs\res_f^\eps\cdot\bs\phi \,\dd \bx\dd t \nonumber
- \T\int_0^T\!\!\int_{\omega}\bs\res_b^\eps\cdot\bs\phi \,\dd x'\dd t 
- \langle \bs\res_s^h,\bs\psi\rangle\,.
\end{align}
for all test functions $(\bs\phi,\bs\psi)\in \W(0,T;\Omega_{\eps,h})$.  

\subsection{Basic error estimate}
Let us first introduce some notation. For an $L^2$-function $\psi\in L^2(0,h)$ we introduce orthogonal 
decomposition (w.r.t.~$L^2(0,h)$-inner product) denoted by
$\psi = \psi^\even + \psi^\odd$, where $\psi^\even$ and $\psi^\odd$ 
denote the even and the odd part of $\psi$, respectively. Furthermore, functions
$\psi\in L^2(\Omega_h)$ will be considered as $\psi\in L^2(\omega;L^2(0,h))$ and the orthogonal
decomposition $\psi = \psi^\even + \psi^\odd$ will be performed in a.e.~point of $\omega$. 

Our key result for proving Theorem \ref{tm:EE} is an energy type estimate for errors, which
we derive from equation (\ref{eq:weak_fsi_error}) based on a careful selection of test functions.
\begin{proposition}\label{prop:EE}
Let us assume that the fluid volume force verifies assumption (A2) 
then for a.e.~$t\in(0,T)$ we have
\begin{align}
\frac{\vro_f}{4}\int_{\Omega_\eps}|\bs e_f^\eps(t)|^2\dd\bx + \nonumber
\frac{\eta\T}{2}\int_0^t\!\!\int_{\Omega_\eps}|\nabla\bs e_f^\eps|^2\,\dd \bx\dd s
+\frac{\vro_s^h\T^{-2}}{4}\int_{\Omega_h}\left((\pa_t e_{s,\alpha}^\even(t))^2 
+ (\pa_t e_{s,3}^\odd(t))^2\right)\dd\bx \\ 
+ \int_{\Omega_h}\left(\mu^h\left|\sym\nabla ( e_{s,1}^\even,  e_{s,2}^\even, e_{s,3}^\odd)(t)\right|^2
+ \frac{\lambda^h}{2}\left|\diver( e_{s,1}^\even,  e_{s,2}^\even, e_{s,3}^\odd)(t)\right|^2 
\right)\,\dd \bx \label{ineq:error_est2}\\
 \leq C\T\eps^3(h^\gamma + h^{4\gamma - 2\kappa + 4})\,.\nonumber
\end{align}
\end{proposition} 
\begin{proof}[Proof of Proposition \ref{prop:EE}] 
Since the elasticity equations appear to be more delicate for the analysis, we first choose
\begin{equation}\label{def:test_psi}
\bs\psi = \T^{-1}(\pa_t e_{s,1}^\even, \pa_t e_{s,2}^\even,\pa_t e_{s,3}^\odd)\,,
\end{equation}   
where superscripts denote even and odd components of the orthogonal decomposition
with respect to the variable $(x_3 - h/2)$.  Observe from (\ref{def:ua_h}) 
that, up to spatially constant translations, components of the approximate displacement $\bs\ua^h$ 
are respectively odd, odd and even with respect to $(x_3-h/2)$.
The idea of using this particular test function comes from the fact that such $\bs\psi$
annihilates a large part of the structure residual term $\bs\res_s^h$ on the right hand side in (\ref{eq:weak_fsi_error})
and the rest can be controlled (cf.~estimate (\ref{ineq:res_struct}) below).  

In order to match interface values of $\bs\psi$, the fluid test function $\bs\phi$ 
will be accordingly corrected fluid error, i.e.~we take
\begin{equation}   
\bs\phi = \bs e_f^\eps + \bs\varphi\,,
\end{equation}
where the correction $\bs\varphi$ satisfies
\begin{align}
\diver\bs\varphi &= 0\quad\text{on }\ \Omega_\eps\times(0,T)\,,\label{eq:divcorr1}\\ 
\left.\varphi_\alpha\right|_{\omega\times(0,T)} &= -\T^{-1}\left.\pa_tu_\alpha^\odd\right|_{\omega\times(0,T)},\\
\left.\phi_3\right|_{\omega\times(0,T)} &= -\T^{-1}\left.\pa_te_{s,3}^\even\right|_{\omega\times(0,T)}\,,\\
\left.\bs\varphi\right|_{\{x_3=-\eps\}\times(0,T)} &= 0\,, \label{eq:divcorr3}
\end{align}
and $\bs\varphi(\cdot,t)$ is $\omega$-periodic for every $t\in(0,T)$. 
This choice of $\bs\varphi$ ensures
the kinematic boundary condition $\bs\phi = \bs\psi$ a.e.~on $\omega\times(0,T)$. 
Moreover, the corrector $\bs\varphi$ satisfies the uniform bound
\begin{lemma}\label{lemma:corr_phi}
\begin{equation}\label{ineq:corr_phi}
\|\nabla\bs\varphi\|_{L^\infty(0,T;L^2(\Omega_\eps))} \leq C\eps^{5/2}\,,
\end{equation}
where $C>0$ is independent of $\bs\varphi$ and $\eps$.
\end{lemma}
\begin{proof}
Following \cite[Lemma 9]{MaPa01}, solution $\bs\varphi$ of the problem (\ref{eq:divcorr1})--(\ref{eq:divcorr3}) 
can be estimated as
\begin{align}\label{ineq:test_corr_bound}
\|\nabla\bs\varphi\|_{L^\infty(0,T;L^2(\Omega_\eps))} 
& \leq C\T^{-1}\left(\sum_{\alpha=1}^2\frac{1}{\sqrt{\eps}}\|\pa_tu^\odd_\alpha\|_{L^\infty(0,T;L^2(\omega))} 
+ \|\pa_te_{s,3}^\even\|_{L^\infty(0,T;L^2(\omega))}\right) \,, 
\end{align}
where $C>0$ is independent of $\eps$ and $t$.
Let us now estimate the right hand side of (\ref{ineq:test_corr_bound}). 
First, employing inequalities on thin domains: the trace inequality from \cite{LeMu11},
the Poincar\'e and the Korn inequality from Proposition \ref{Poincare}, respectively,   
we find
\begin{align}\label{ineq:corr_phi1}
\sum_{\alpha=1}^2\|\pa_tu^\odd_\alpha\|^2_{L^\infty(0,T;L^2(\omega))}
& \leq C\left(\frac{1}{h}\|\pa_t\bu^\odd\|_{L^\infty(0,T;L^2(\Omega_h))}^2 + 
h\|\nabla\pa_t\bu^\odd\|^2_{L^\infty(0,T;L^2(\Omega_h))} \right) \\ \nonumber
&\leq C\left(\frac{1}{h} \|\nabla\pa_t\bu^\odd\|_{L^\infty(0,T;L^2(\Omega_h))}^2 + 
\frac{1}{h} \|\sym\nabla\pa_t\bu^\odd\|_{L^\infty(0,T;L^2(\Omega_h))} \right)\\ \nonumber
&\leq \frac{C}{h^3}\|\sym\nabla\pa_t\bu^\odd\|_{L^\infty(0,T;L^2(\Omega_h))}^2
\leq \frac{C}{h^3}\|\sym\nabla\pa_t\bu\|_{L^\infty(0,T;L^2(\Omega_h))}^2 \\ \nonumber 
&\leq C\T\eps^3h^{\kappa-3}\,.
\end{align}
In the latter inequality we used the time rescaled higher-order energy estimate 
(\ref{ineq:energEHH}).

Utilizing the Griso decomposition for the third component
\begin{align*}
\pa_te_{s,3}^\even = \pa_t(w_3^h+\tilde u_3^\even) - h^{\kappa-3}\pa_tw_3 
= \pa_t u_3^h - \pa_t\tilde u_3^\odd - h^{\kappa-3}\pa_tw_3
\end{align*}
and estimating the second term by using the trace inequality \cite{LeMu11}, we have 
\begin{align*}
\|\pa_te_{s,3}^\even\|_{L^\infty(0,T;L^2(\omega))} ^2
&\leq \frac{C}{h}\|\pa_t(u_3^h - \tilde u_3^\odd)\|^2_{L^\infty(0,T;L^2(\Omega_h))}  
+ Ch\|\nabla\pa_t(u_3^h - \tilde u_3^\odd)\|^2_{L^\infty(0,T;L^2(\Omega_h))}  \\
&\quad + h^{2\kappa-6}\|\pa_tw_3\|_{L^\infty(0,T;L^2(\omega))} \\
&\leq \frac{C}{h}\|\pa_tu_3^h\|^2_{L^\infty(0,T;L^2(\Omega_h))} 
+ Ch\|\nabla\pa_tu_3^h\|^2_{L^\infty(0,T;L^2(\Omega_h))} 
 + \frac{C}{h}\|\pa_t\tilde u_3^\odd\|^2_{L^\infty(0,T;L^2(\Omega_h))} \\
&\quad + Ch\|\nabla\pa_t\tilde u_3^\odd\|^2_{L^\infty(0,T;L^2(\Omega_h))} 
+ h^{2\kappa-6}\|\pa_tw_3\|_{L^\infty(0,T;L^2(\omega))}\,.
\end{align*}

Performing the Griso decomposition of the structure velocity $\pa_t\bu^h$ and 
employing the Griso estimates (cf.~\ref{app:Griso_estimate}), the time rescaled 
higher-order energy inequality 
(\ref{ineq:energEHH}) implies
\begin{align*}
\frac{1}{h^2}\|\pa_t\tilde \bu^h\|^2_{L^\infty(0,T;L^2(\Omega_h))} + 
\|\nabla\pa_t\tilde \bu^h\|^2_{L^\infty(0,T;L^2(\Omega_h))} \leq C\T \eps^3 h^\kappa
\end{align*}
and $\|\pa_tw_3\|_{L^\infty(0,T;L^2(\omega))}\leq C$.
Using the latter together with the Poincar\'e inequality we further estimate
\begin{align}\label{ineq:pate3}
\|\pa_te_{s,3}^\even\|_{L^\infty(0,T;L^2(\omega))}^2
& \leq \frac{C}{h}\|\nabla\pa_tu_3^h\|^2_{L^\infty(0,T;L^2(\Omega_h))} 
+ \frac{C}{h^2}\esssup_{t\in(0,T)}\left|\int_{\Omega_h}\pa_t u_3^h\dd \bx \right|^2\\
&\qquad + C\T\eps^3h^{\kappa+1} + Ch^{2\kappa-6}\,. \nonumber
\end{align}
In order to conclude the estimate (\ref{ineq:corr_phi}) we need one more result.

\begin{lemma} The mean values of the vertical structure displacement and velocity satisfy
\begin{equation*}
\esssup_{t\in(0,T)}\left|\int_{\Omega_h}u_3^h\dd \bx\right| + 
\esssup_{t\in(0,T)}\left|\int_{\Omega_h}\pa_tu_3^h\dd \bx\right| \leq Ch^\kappa\,.
\end{equation*}
\end{lemma}
\begin{proof}
Let us define 
\begin{equation*}
U_3^h(x_3,t):=\int_\omega u_3^h(x',x_3,t)\dd x'\,,\quad x_3\in(0,h)\,,\ t\in(0,T)\,.
\end{equation*}
Then for a.e.~$t \in(0,T)$, using the Cauchy-Schwarz inequality we have     
\begin{align*}
\left|\int_{\Omega_h}u_3^h(t)\dd \bx\right| = \int_0^h U_3^h(x_3,t)\dd x_3 \leq \sqrt h\|U_3^h(t)\|_{L^2(0,h)}\,.
\end{align*}
Since $U_3^h(0,t) = \int_\omega u_3^h(x',0,t)\dd x' = 0$, the Poincare inequality on $(0,h)$ gives
\begin{equation*}
\left|\int_{\Omega_h}u_3^h(t)\dd \bx\right| \leq Ch^{3/2}\|\pa_3 U_3^h(t)\|_{L^2(0,h)}\,.
\end{equation*}
Next, using the Jensen's inequality and the time rescaled energy estimate (\ref{ineq:energEH}) we find
\begin{align*}
\|\pa_3U_3^h(t)\|_{L^2(0,h)}^2 = \int_0^h\left( \pa_3U_3^h(t)\right)^2\dd x_3 = 
\int_0^h\left( \int_\omega \pa_3 u_3^h(x',x_3,t)\dd x'\right)^2\dd x_3\\
\leq \int_{\Omega_h}\left(\pa_3 u_3^h(x',x_3,t) \right)^2\dd \bx \leq C\T\eps^3h^\kappa\,.
\end{align*}
Therefore, employing the latter inequality together with relation (\ref{eq:rel_tau}), we conclude 
\begin{equation*}
\esssup_{t\in(0,T)}\left|\int_{\Omega_h}u_3^h\dd \bx\right| 
\leq Ch^{3/2}\|\pa_3 U_3^h\|_{L^\infty(0,T;L^2(0,h))}\leq Ch^\kappa\,.
\end{equation*}
Due to the time rescaled higher-order energy estimate (\ref{ineq:energEHH}), 
which is of the same type as (\ref{ineq:energEH}), the analogous conclusion
can be performed also for $\pa_t u_3^h$.
\end{proof}
Going back to (\ref{ineq:pate3}) and applying the previous lemma with the Korn inequality and
the time rescaled energy estimate (\ref{ineq:energEHH}), we obtain
\begin{align}\label{ineq:corr_phi2}
\|\pa_te_{s,3}^\even\|_{L^\infty(0,T;L^2(\omega))}^2
& \leq \frac{C}{h^3}\|\sym\nabla\pa_tu_3^h\|^2_{L^\infty(0,T;L^2(\Omega_h))} + Ch^{2\kappa - 2} + Ch^{2\kappa-6} \\
&\leq  C\T\eps^3h^{\kappa-3} + Ch^{2\kappa-6}\,.\nonumber
\end{align}
Combining (\ref{ineq:corr_phi1}) and (\ref{ineq:corr_phi2}), from estimate (\ref{ineq:test_corr_bound}) and relation
(\ref{eq:rel_tau}) we conclude 
\begin{equation}\label{ineq:test_corr_bound2}
\|\nabla\bs\varphi\|_{L^\infty(0,T;L^2(\Omega_h))}^2 
\leq C\left(\T^{-1}\eps^2h^{\kappa-3} + \T^{-1}\eps^3h^{\kappa+1} + \T^{-2}h^{2\kappa - 6} \right) 
\leq C\eps^5\,,
\end{equation}
which finishes the proof of Lemma \ref{lemma:corr_phi}. 
\end{proof}

Now we continue with the proof of Proposition \ref{prop:EE}.
Utilizing the above constructed test functions $(\bs\phi,\bs\psi)$ in the variational equation
(\ref{eq:weak_fsi_error}) and using the orthogonality property 
of the decomposition to even and odd functions with respect to the variable $(x_3 - h/2)$, then 
for a.e.~$t\in(0,T)$ we have
\begin{align*}
\frac{\vro_f}{2}\int_{\Omega_\eps}|\bs e_f^\eps(t)|^2\dd\bx +
2\eta\T\int_0^t\!\!\int_{\Omega_\eps}|\sym\nabla\bs e_f^\eps|^2\,\dd \bx\dd s
+\frac{\vro_s^h\T^{-2}}{2}\int_{\Omega_h}\left((\pa_t e_{s,\alpha}^\even)^2 
+ (\pa_t e_{s,3}^\odd)^2\right)\dd\bx \\ 
+ \int_{\Omega_h}\left(\mu^h\left|\sym\nabla ( e_{s,1}^\even,  e_{s,2}^\even, e_{s,3}^\odd)\right|^2
+ \frac{\lambda^h}{2}\left|\diver( e_{s,1}^\even,  e_{s,2}^\even, e_{s,3}^\odd)\right|^2
\right)(t)\,\dd \bx
= -\vro_f\int_{\Omega_\eps}\bs e^\eps_f(t)\cdot\bs\varphi\,\dd \bx \\
- 2\eta\T\int_0^t\!\!\int_{\Omega_\eps}\sym\nabla\bs e_f^\eps:\sym\nabla\bs\varphi\,\dd \bx\dd s
+\T\int_0^t\!\!\int_{\Omega_\eps}f_3^\eps (e_{f,3}^\eps + \varphi_3) \,\dd \bx\dd s \\
- \T\int_0^t\!\!\int_{\Omega_\eps}\bs\res_f^\eps\cdot(\bs e_f^\eps + \bs\varphi) \,\dd \bx\dd s 
- \T\int_0^t\!\!\int_{\omega}\bs\res_b^\eps\cdot(\bs e_f^\eps + \bs\varphi) \,\dd x'\dd s
- \langle \bs\res_s^h,\bs\psi\rangle\,.
\end{align*}
Using the Cauchy-Schwarz and the Young inequality together with inequalities 
from Proposition \ref{Poincare}
we estimate the right hand side of the latter equation as follows:
\begin{align}
\frac{\vro_f}{4}\int_{\Omega_\eps}|\bs e_f^\eps(t)|^2\dd\bx + \nonumber
\eta\T\int_0^t\!\!\int_{\Omega_\eps}|\sym\nabla\bs e_f^\eps|^2\,\dd \bx\dd s
+\frac{\vro_s^h\T^{-2}}{2}\int_{\Omega_h}\left((\pa_t e_{s,\alpha}^\even)^2
+ (\pa_t e_{s,3}^\odd)^2\right)(t)\dd\bx \\ 
+ \int_{\Omega_h}\left(\mu^h\left|\sym\nabla ( e_{s,1}^\even,  e_{s,2}^\even, e_{s,3}^\odd)\right|^2
+ \frac{\lambda^h}{2}\left|\diver( e_{s,1}^\even,  e_{s,2}^\even, e_{s,3}^\odd)\right|^2 \nonumber
\right)(t)\,\dd \bx 
\leq \frac{\vro_f}{4}\int_{\Omega_\eps}|\bs \varphi(t)|^2\dd\bx \\ \label{ineq:error_est1}
 + \eta\T\int_0^t\!\!\int_{\Omega_\eps}|\sym\nabla\bs \varphi|^2\,\dd \bx\dd s
+ \T\int_0^t\|f_3^\eps\|_{L^2(\Omega_\eps)}\left(\|e_{f,3}^\eps\|_{L^2(\Omega_\eps)} 
+ \eps\|\nabla\varphi_3\|_{L^2(\Omega_\eps)}\right)\dd s \\
+ C\T\int_0^t\eps\|\bs\res_f^\eps\|_{L^2(\Omega_\eps)}\left(\|\nabla\bs e_f^\eps\|_{L^2(\Omega_\eps)}
+ \|\nabla\bs\varphi\|_{L^2(\Omega_\eps)}\right)\dd s \nonumber
+ \left|\T\int_0^t\!\!\int_{\omega}\bs\res_b^\eps\cdot \bs e_f^\eps\,\dd s\right|\\
+ \T\int_0^t\eps^{1/2}\|\bs\res_b^\eps\|_{L^2(\omega)}
\|\nabla\bs\varphi\|_{L^2(\Omega_\eps)} \dd s     
+ |\langle \bs\res_s^h,\bs\psi\rangle| \,.\nonumber
\end{align}

The right hand side in (\ref{ineq:error_est1}) is further estimated term by term as follows.
The first two terms are bounded by 
\begin{align}
\frac{\vro_f}{4}\int_{\Omega_\eps}|\bs \varphi(t)|^2\dd\bx 
 + \eta\T\int_0^t\!\!\int_{\Omega_\eps}|\sym\nabla\bs \varphi|^2\,\dd \bx\dd s 
 \leq C\eps^7 + C\T\eps^5 \leq C\T\eps^5\,.
\end{align}
Higher-order energy estimate (\ref{ineq:energEHHS}) directly provides
\begin{align}\label{ineq:HigherOrder}
\|\nabla\pa_{\alpha}\bv^{\eps}\|^2_{L^2(0,T;L^2(\Omega_{\eps}))}\leq C\eps^3\,.
\end{align}
Then the Poincar\'e inequality on thin domains implies
\begin{align*}
\|\pa_{\alpha}\bv^{\eps}\|_{L^2(0,T;L^2(\Omega_{\eps}))}\leq C\eps^{5/2}\,.
\end{align*}
Another application of the Poincar\'e inequality combined with divergence free 
condition yields
\begin{align}\label{ineq:VertEstimate}
\|v^{\eps}_3\|_{L^2(0,T;L^2(\Omega_{\eps}))}\leq C\eps^{7/2}\,.
\end{align}
The latter trivially implies
\begin{align*}
\|e_{f,3}^{\eps}\|_{L^2(0,T;L^2(\Omega_{\eps}))} = 
\|v^\eps_3 - \va_3^\eps\|_{L^2(0,T;L^2(\Omega_{\eps}))} \leq C\eps^{7/2}\,.
\end{align*}
Therefore, the force term can be bounded as
\begin{align}
 \T\int_0^t\|f_3^\eps\|_{L^2(\Omega_\eps)}\left(\|e_{f,3}^\eps\|_{L^2(\Omega_\eps)} 
 + \eps\|\nabla\varphi_3\|_{L^2(\Omega_\eps)}\right)\dd s \leq C\T\eps^4\,.
\end{align}
For the fluid residual term we employ the apriori estimates to conclude
\begin{align}
C\T\int_0^t\eps\|\bs\res_f^\eps\|_{L^2(\Omega_\eps)}\left(\|\nabla\bs e_f^\eps\|_{L^2(\Omega_\eps)}
+ \|\nabla\bs\varphi\|_{L^2(\Omega_\eps)}\right)\dd s \leq C\T\eps^4\,.
\end{align}
For the boundary residual term, which is only $O(\eps)$ in the leading order, we invoke the Griso 
decomposition to conclude that 
\begin{equation*}
\left.\bs e_f^\eps\right|_\omega = \left.\T^{-1}\pa_t\bs u^h\right|_\omega - \eps^2\pa_ta_\alpha
\end{equation*}
is dominantly constant on $\omega$. Due to the interface condition 
$\displaystyle \int_\omega \pa_3 v_\alpha \, \dd z' = 0$, the leading order term vanishes and the rest
can be controlled as
\begin{align}
\left|\T\int_0^t\!\!\int_{\omega}\bs\res_b^\eps\cdot \bs e_f^\eps\,\dd s\right|
 + \T\int_0^t\eps^{1/2}\|\bs\res_b^\eps\|_{L^2(\omega)}
\|\nabla\bs\varphi\|_{L^2(\Omega_\eps)} \dd s \leq C\T\eps^4\,.
\end{align}
Finally, due to the orthogonality properties and integrating by parts in time,   
for the structure residual term we have
\begin{align*}
\langle \bs\res_s^h,\bs\psi\rangle &= 
\vro_s^h\T^{-2}\int_0^t\!\!\int_{\Omega_h}h^{\kappa-3-\gamma}(\pa_ta_1\pa_{tt} e_{s,1}^\even 
+ \pa_ta_2\pa_{tt} e_{s,2}^\even)\dd\bx\dd s \\
&= \vro_s^h\T^{-2}\int_{\Omega_h}h^{\kappa-3-\gamma}(\pa_ta_1\pa_{t} e_{s,1}^\even 
+ \pa_ta_2\pa_{t} e_{s,2}^\even)(t)\dd\bx \\
&\qquad - \vro_s^h\T^{-2}\int_0^t\!\!\int_{\Omega_h}h^{\kappa-3-\gamma}(\pa_{tt}a_1\pa_{t} e_{s,1}^\even 
+ \pa_{tt}a_2\pa_{t} e_{s,2}^\even)\dd\bx\dd s\,.
\end{align*}
The latter can be estimated as
\begin{align}\label{ineq:res_struct}
\left|\langle \bs\res_s^h,\bs\psi\rangle\right| \leq C\vro_s^h\T^{-2}h^{2\kappa - 6 - 2\gamma + 1} 
+ \frac{\vro_s^h\T^{-2}}{4}\int_{\Omega_h}\left((\pa_{t} e_{s,1}^\even)^2 + (\pa_{t} e_{s,2}^\even)^2 \right)(t)\dd\bx\\
+ \vro_s^h\T^{-2}\int_0^t\!\!\int_{\Omega_h}\left((\pa_{t} e_{s,1}^\even)^2 
+ (\pa_{t} e_{s,2}^\even)^2 \right)\dd\bx\dd s\,. \nonumber
\end{align} 
Going back to (\ref{ineq:error_est1}) and employing previously established bounds together with the
Gr\"onwall inequality, we find
\begin{align}
\frac{\vro_f}{4}\int_{\Omega_\eps}|\bs e_f^\eps(t)|^2\dd\bx + \nonumber
\eta\T\int_0^t\!\!\int_{\Omega_\eps}|\sym\nabla\bs e_f^\eps|^2\,\dd \bx\dd s
+\frac{\vro_s^h\T^{-2}}{4}\int_{\Omega_h}\left((\pa_t e_{s,\alpha}^\even)^2 
+ (\pa_t e_{s,3}^\odd)^2\right)\dd\bx \\ 
+ \int_{\Omega_h}\left(\mu^h\left|\sym\nabla ( e_{s,1}^\even,  e_{s,2}^\even, e_{s,3}^\odd)\right|^2
+ \frac{\lambda^h}{2}\left|\diver( e_{s,1}^\even,  e_{s,2}^\even, e_{s,3}^\odd)\right|^2 
\right)\,\dd \bx \label{ineq:error_est2}\\
\leq C\T(\eps^4 + h^{7\gamma - 2\kappa + 4 }) = C\T\eps^3(h^\gamma + h^{4\gamma - 2\kappa + 4})\,.\nonumber
\end{align}

In order to finish with the proof of Proposition \ref{prop:EE}, we still need to estimate tangential
fluid errors on the interface. Namely,
\begin{align*}
\int_0^t\!\!\int_{\Omega_\eps}|\sym\nabla\bs e_f^\eps|^2\,\dd \bx\dd s = 
\frac12\int_0^t\!\!\int_{\Omega_\eps}|\nabla\bs e_f^\eps|^2\,\dd \bx\dd s 
+ \int_0^t\!\!\int_{\omega}\left(e_{f,1}^\eps\pa_1e_{f,3}^\eps + e_{f,2}^\eps\pa_2e_{f,3}^\eps\right)\dd x'\dd s\,.
\end{align*}
The interface terms are then estimated as
\begin{align*}
\left| \int_0^t\!\!\int_{\omega}e_{f,\alpha}^\eps\pa_\alpha e_{f,3}^\eps\dd x'\dd s \right|
&\leq \eps\|\pa_3e_{f,\alpha}^\eps\|_{L^2(0,t;L^2(\Omega_\eps))}
\|\pa_3\pa_\alpha e_{f,3}^\eps\|_{L^2(0,t;L^2(\Omega_\eps))}\\
& \leq \frac{\eta}{4}\|\pa_3e_{f,\alpha}^\eps\|_{L^2(0,t;L^2(\Omega_\eps))}^2 
+ C\eps^2\|\pa_3\pa_\alpha e_{f,3}^\eps\|_{L^2(0,t;L^2(\Omega_\eps))}^2\\
&\leq \frac{\eta}{4}\|\pa_3e_{f,\alpha}^\eps\|_{L^2(0,t;L^2(\Omega_\eps))}^2 
+ C\eps^5\,,
\end{align*}
where the last bound follows from (\ref{ineq:HigherOrder}). Employing this estimate into (\ref{ineq:error_est2})
we arrive to
\begin{align*}
\frac{\vro_f}{4}\int_{\Omega_\eps}|\bs e_f^\eps(t)|^2\dd\bx + \nonumber
\frac{\eta\T}{4}\int_0^t\!\!\int_{\Omega_\eps}|\nabla\bs e_f^\eps|^2\,\dd \bx\dd s
+\frac{\vro_s^h\T^{-2}}{4}\int_{\Omega_h}\left((\pa_t e_{s,\alpha}^\even)^2 
+ (\pa_t e_{s,3}^\odd)^2\right)\dd\bx \\ 
+ \int_{\Omega_h}\left(\mu^h\left|\sym\nabla ( e_{s,1}^\even,  e_{s,2}^\even, e_{s,3}^\odd)\right|^2
+ \frac{\lambda^h}{2}\left|\diver( e_{s,1}^\even,  e_{s,2}^\even, e_{s,3}^\odd)\right|^2 
\right)\,\dd \bx \leq  C\T\eps^3(h^\gamma + h^{4\gamma - 2\kappa + 4})\,,
\end{align*}
which finishes the proof of Proposition \ref{prop:EE}.
\end{proof}

\subsection{Error estimates for fluid velocities}

Let us now start with the proof of Theorem \ref{tm:EE}.
Proposition \ref{prop:EE} directly implies
\begin{align*}
\|\nabla \bs e_{f}^\eps\|_{L^2(0,T;L^2(\Omega_\eps))}^2 \leq C\eps^3(h^\gamma + h^{4\gamma - 2\kappa + 4})\,,
\end{align*}
while the Poincar\'e inequality then gives
\begin{equation*}
\|e_{f,\alpha}^\eps\|_{L^2(0,T;L^2(\Omega_\eps))}^2 \leq C\eps^5(h^\gamma + h^{4\gamma - 2\kappa + 4})\,.
\end{equation*}
The latter implies the desired error estimate
\begin{equation}
\|e_{f,\alpha}^\eps\|_{L^2(0,T;L^2(\Omega_\eps))} 
\leq C\eps^{5/2}h^{\min\{\gamma/2,\,(2\gamma-\kappa+2)_+\}}\,,\quad \alpha=1,2\,,
\end{equation} 
where $(2\gamma-\kappa+2)_+ = \max\{2\gamma-\kappa+2,\,0\}$.

\subsection{Error estimates for structure displacements}
For the structure displacement error, Proposition \ref{prop:EE} provides only
\begin{equation*}
\|\sym\nabla ( e_{s,1}^\even,  e_{s,2}^\even, e_{s,3}^\odd)\|_{L^\infty(0,T;L^2(\Omega_h))}^2
\leq C\left(h^{2\kappa + \gamma - 3} + h^{4\gamma+1}\right)\,.
\end{equation*}
According to the Griso decomposition for $u_\alpha^h$, we have
\begin{equation*}
e_{s,\alpha}^\even = w_\alpha^h + \tilde u_\alpha^\even - h^{\kappa-3-\gamma}a_\alpha\,,
\quad \alpha=1,2\,.
\end{equation*}
where $\tilde u_\alpha^\even$ denotes the even part of the warping $\tilde{u}_\alpha$.
Employing now the Griso decomposition for the error $e_{s,\alpha}^\even$ together with the
Griso estimate for the corresponding elementary plate displacement, we have
\begin{align*}
\|\sym\nabla'(w_1^h + \frac1h\int_0^h\tilde u_1^\even \dd x_3,w_2^h+ \frac1h\int_0^h\tilde u_2^\even \dd x_3)\|_{L^\infty(0,T;L^2(\Omega_h))}^2 
&\leq \|\sym\nabla'(e_{s,1}^\even,e_{s,2}^\even)\|_{L^\infty(0,T;L^2(\Omega_h))}^2\\
&\leq C\left(h^{2\kappa + \gamma - 3} + h^{4\gamma+1}\right)\,,
\end{align*} 
i.e.
\begin{equation*}
\|\sym\nabla'(w_1^h + \frac1h\int_0^h\tilde u_1^\even \dd x_3,w_2^h+ \frac1h\int_0^h\tilde u_2^\even \dd x_3)\|_{L^\infty(0,T;L^2(\omega))}^2 
\leq C\left(h^{2\kappa + \gamma - 4} + h^{4\gamma}\right)\,.
\end{equation*}
According to the Korn inequality and the Griso estimate for the warping terms $\tilde u_\alpha^\even$, 
for the sequence of spatially constant functions 
$(a_\alpha^h)\subset L^\infty(0,T)$ we have
\begin{equation}\label{ineq:w1w2}   
\|(w_1^h-a_1^h,w_2^h-a_2^h)\|_{L^\infty(0,T;L^2(\omega))}^2 
\leq C\left(h^{2\kappa + \gamma - 4} + h^{4\gamma} + h^{2\kappa-2}\right)\,.  
\end{equation}

Again the Griso decomposition and a priori estimate of the warping terms $\tilde{u}_\alpha^h$ imply
\begin{align*}
\|e_{s,1}^h - w_1^h + h^{\kappa-3-\gamma}a_1 - (x_3 - \frac{h}{2})(h^{\kappa-3}\pa_1w_3 + r_2^h)\|_{L^\infty(0,T;L^2(\Omega_h))}^2 &= 
\|\tilde{u}_1^h\|_{L^\infty(0,T;L^2(\Omega_h))}^2 \\ & \leq Ch^{2\kappa-1}\,,\\
\|e_{s,2}^h - w_2^h +  h^{\kappa-3-\gamma}a_2 - (x_3 - \frac{h}{2})(h^{\kappa-3}\pa_2w_3 - r_1^h)\|_{L^\infty(0,T;L^2(\Omega_h))}^2 &= 
\|\tilde{u}_2^h\|_{L^\infty(0,T;L^2(\Omega_h))}^2\\ &\leq Ch^{2\kappa-1}\,.
\end{align*}
Using the triangle inequality and estimate (\ref{ineq:w1w2}) we have
\begin{equation}\label{ineq:disperror_alpha}
\|e_{\alpha,s}^h\|_{L^\infty(0,T;L^2(\Omega_h))}^2 
\leq Ch^{2\kappa-3}\left(h^{2} + h^{\gamma} + h^{4\gamma + 4 - 2\kappa} \right)
+ h\|a_\alpha^h - h^{\kappa-3-\gamma}a_\alpha\|_{L^\infty(0,T)}^2\,,
\end{equation}
provided the following lemma holds true.
\begin{lemma}
\begin{align}\label{ineq:daw3ra}
\|h^{\kappa-3}\pa_1w_3 + r_2^h\|_{L^\infty(0,T;L^2(\Omega_h))}^2 &+ 
\|h^{\kappa-3}\pa_2w_3 - r_1^h\|_{L^\infty(0,T;L^2(\Omega_h))}^2 \\
&\qquad \leq Ch^{2\kappa - 5}\left(h^{2} + h^{\gamma} + h^{4\gamma + 4 - 2\kappa} \right)\,.
\nonumber
\end{align}
\end{lemma}

\begin{proof}
In order to prove (\ref{ineq:daw3ra}) we first employ the test function on the structure part
$\bs\psi = \T^{-1}\pa_t\bs e_s^{\odd}$ in (\ref{eq:weak_fsi_error}). For the fluid part 
we take the test function
\begin{equation*}
\bs\phi = \bs e_f^\eps + \bs\varphi\,,
\end{equation*}
where the correction $\bs\varphi$ satisfies
\begin{align*}
\diver\bs\varphi &= 0\quad\text{on }\ \Omega_\eps\times(0,T)\,,\\ 
\varphi_\alpha|_\omega &= \T^{-1}\left.\left(\frac{h^{\kappa-2}}{2}\pa_\alpha w_3 - \pa_te_{s,\alpha}^\even\right)\right|_\omega\,,\ 
\left.\phi_3\right|_\omega = -\left.\T^{-1}\pa_te_{s,3}^\even\right|_\omega\,,
\end{align*}
$\left.\bs\varphi\right|_{\{x_3=-\eps\}} = 0$ and $\bs\varphi$ is periodic on the lateral boundaries. 
The estimate on $\bs\varphi$ now reads
\begin{equation*}
\|\nabla\bs\varphi\|_{L^\infty(0,T;L^2(\Omega_\eps))}^2 
\leq C\T^{-2}\left(\frac{1}{\eps}\left\|\frac{h^{\kappa-2}}{2}\pa_\alpha w_3 
- \pa_te_{s,\alpha}^\even\right\|_{L^\infty(0,T;L^2(\omega))}^2 + 
\|\pa_te_{s,3}^\even\|_{L^\infty(0,T;L^2(\omega))}^2\right)\,.
\end{equation*}
The last term is already estimated above with $O(\eps^6)$. Therefore, 
using the trace and Korn inequalities on thin domains together with estimates
(\ref{ineq:error_est2}) and (\ref{ineq:energEHH}) we have
\begin{align*}
\|\nabla\bs\varphi\|_{L^\infty(0,T;L^2(\Omega_\eps))}^2 
&\leq C\left(h^{5\gamma+2} + \frac{\T \eps^4}{\eps h}h^\kappa 
+ \frac{\T^{-2}}{\eps h}\|\sym\nabla\pa_t\bu^h\|_{L^\infty(0,T;L^2(\Omega_h))}^2 + h^{6\gamma} \right)\\
&\leq C\left(h^{5\gamma+2} + h^{2\kappa-4} + h^{6\gamma} \right)\leq C(\eps^5 + h^{2\kappa-4})\,. 
\end{align*}
Using the above test functions in the weak form for the errors (\ref{eq:weak_fsi_error}) and estimating
like in (\ref{ineq:error_est1}) we find
\begin{align}
\frac{\vro_f}{4}\int_{\Omega_\eps}|\bs e_f^\eps(t)|^2\dd\bx &+ \nonumber
\frac{\eta\T}{2}\int_0^t\!\!\int_{\Omega_\eps}|\sym\nabla\bs e_f^\eps|^2\,\dd \bx\dd s
+ \frac{\vro_s^h\T^{-2}}{4}\int_{\Omega_h}\left|\pa_t \bs e_{s}^\odd(t)\right|^2\dd\bx \\ 
&\quad + \int_{\Omega_h}\left(\mu^h\left|\sym\nabla \bs e_{s}^\odd(t)\right|^2
+ \frac{\lambda^h}{2}\left|\diver\bs e_{s}^\odd(t)\right|^2 
\right)\,\dd \bx \nonumber\\
&\quad \leq C\T\left(\eps^4 + h^{2\kappa-4} + \eps^{3/2}(\eps^5 + h^{2\kappa-4})^{1/2}\right)\label{ineq:error_est3} \\ \nonumber
&\qquad + C\T\eps^3 + \frac{\lambda^h}{4}\int_{\Omega_h}\left|\diver\bs e_{s}^\odd(t)\right|^2\dd\bx
 + \frac{\lambda^h}{2}\int_0^t\!\!\int_{\Omega_h}\left|\diver\bs e_{s}^\odd\right|^2\dd\bx\dd s\,.\nonumber
\end{align}
The last line of the above inequality arises from estimating the structure residual term
$\T^{-1}\langle\bs\res_s^h,\pa_t\bs e_s^{\odd}\rangle$. 
Therefore, employing the Gr\"onwall lemma we close the estimate (\ref{ineq:error_est3})
with
\begin{align}
\frac{\vro_f}{4}\int_{\Omega_\eps}|\bs e_f^\eps(t)|^2\dd\bx &+ \nonumber
\frac{\eta\T}{2}\int_0^t\!\!\int_{\Omega_\eps}|\sym\nabla\bs e_f^\eps|^2\,\dd \bx\dd s
+ \frac{\vro_s^h\T^{-2}}{4}\int_{\Omega_h}\left|\pa_t \bs e_{s}^\odd(t)\right|^2\dd\bx \\ 
&\quad + \int_{\Omega_h}\left(\mu^h\left|\sym\nabla \bs e_{s}^\odd(t)\right|^2
+ \frac{\lambda^h}{4}\left|\diver\bs e_{s}^\odd(t)\right|^2 
\right)\,\dd \bx \leq C\T\left(\eps^3 + h^{2\kappa-4}\right)\,.
\end{align}
Having this at hand, we conclude
\begin{equation*}
\int_{\Omega_h}(\pa_\alpha e_{s,3}^\odd + \pa_3e_{s,\alpha}^\odd)^2\dd \bx 
\leq Ch^{2\kappa-5}\left(h^2 + h^{2\kappa - 3\gamma - 2}\right)\,,
\quad\alpha = 1,2\,,
\end{equation*}
which is equivalent to
\begin{align*}
\int_{\Omega_h}\left(\left(r_2^h + h^{\kappa-3}\pa_1w_3 + \pa_3\tilde u_1^\odd + \pa_1\tilde u_3^\odd \right)^2
+ \left(-r_1^h + h^{\kappa-3}\pa_2w_3 + \pa_3\tilde u_2^\odd + \pa_2\tilde u_3^\odd \right)^2\right)\dd\bx
\\
\leq Ch^{2\kappa-5}\left(h^2 + h^{2\kappa - 3\gamma - 2}\right)\,.
\end{align*}
From the basic Griso inequality and a priori estimate we have
\begin{align*}
\int_{\Omega_h}\left(\left(\pa_3\tilde u_1^\odd + \pa_1\tilde u_3^\odd \right)^2
+ \left(\pa_3\tilde u_2^\odd + \pa_2\tilde u_3^\odd \right)^2\right)\dd\bx
&\leq \|\nabla\tilde{\bu}^h\|_{L^\infty(0,T;L^2(\Omega_h))}^2\\
&\leq \|\sym\nabla{\bu}^h\|_{L^\infty(0,T;L^2(\Omega_h))}^2
\leq Ch^{2\kappa-3}\,.
\end{align*}
One immediately sees that for $\kappa\geq \max\{2\gamma+1, \frac74\gamma+\frac32\}$ it holds 
$h^{2\kappa-3\gamma-2} \leq h^\gamma + h^{4\gamma+4-2\kappa}$.
Thus, using the triangle inequality we conclude the desired estimate (\ref{ineq:daw3ra}).
\end{proof}
Recall again the Griso estimate for the elementary plate displacement, we have
\begin{equation*}
\int_{\Omega_h}\left((\pa_1w^h_3 + r_2^h)^2 + (\pa_2w^h_3 - r_1^h)^2 \right)\dd\bx \leq Ch^{2\kappa-3}\,,
\end{equation*}
which combined with (\ref{ineq:daw3ra}) implies
\begin{align*}
\|\pa_1w_3^h - h^{\kappa-3}\pa_1w_3\|_{L^\infty(0,T;L^2(\Omega_h))}^2 + 
\|\pa_2w_3^h - h^{\kappa-3}\pa_2w_3\|_{L^\infty(0,T;L^2(\Omega_h))}^2 \\
\leq 
Ch^{2\kappa - 5}\left(h^{2} + h^{\gamma} + h^{4\gamma + 4 - 2\kappa} \right)\,.
\end{align*}
The Poincar\'e inequality gives
\begin{equation*}
\|w_3^h - h^{\kappa-3}w_3\|_{L^\infty(0,T;L^2(\Omega_h))}^2 \leq Ch^{2\kappa-4} + 
Ch^{2\kappa - 5}\left(h^{2} + h^{\gamma} + h^{4\gamma + 4 - 2\kappa} \right)\,,
\end{equation*}
which eventually provides 
\begin{equation}
\|e_{s,3}^h\|_{L^\infty(0,T;L^2(\Omega_h))}^2 \leq 
Ch^{2\kappa - 5}\left(h^1 + h^{\gamma} + h^{4\gamma + 4 - 2\kappa} \right)\,.
\end{equation}

\subsection{Error estimate for the pressure} 
Finally we prove the error estimate for the pressure. We define $e_p^{\eps}=p^{\eps}-\paa^{\eps}$. 
Similarly as in the a priori pressure estimate, the error estimate will be performed in two steps. 
In the first step we estimate zero mean value part of the error $e_p^{\eps}$. 
This is classical and follows directly from the error estimates for the fluid velocity. 

The second step is specific for our problem and is related to the fact that the pressure is unique. 
Let us denote by $\pi_e^{\eps}(t)=\int_{\Omega_{\eps}}e_p^{\eps}(t,.)$ the mean value of the pressure error. 
The test function $\bs\phi\in \mathcal{V}_F(0,T;\Omega_{\eps})$ is constructed such 
that $\diver\bs\phi(t,.)=e_p^{\eps}(t,.)-\pi_e^{\eps}(t)$ and $\bs\phi$ vanishes on the interface.
This can be done in a standard way by using the Bogovskij construction, 
see e.g.~\cite[Section 3.3]{galdi2011introduction}. 
Moreover, the following estimates hold (see e.g.~\cite[Lemma 9]{MaPa01}):
\begin{align*}
\|\bs\phi\|_{L^2(0,T;H^1(\Omega_{\eps})} &\leq \frac{C}{\eps}\|e_p^{\eps}-\pi_e^{\eps}\|_{L^2(0,T;L^2(\Omega_{\eps})}\,,\\
\|\phi_3\|_{L^2(0,T;L^2(\Omega_{\eps})} &\leq C\|e_p^{\eps}-\pi_e^{\eps}\|_{L^2(0,T;L^2(\Omega_{\eps})}\,.
\end{align*}
By construction $(\bs\phi,0)$ is and admissible test function for error formulation and therefore 
we get the following estimate:
\begin{align*}
\int_0^T\!\!\int_{\Omega_{\eps}}(e_p^{\eps}-\pi_e^{\eps})^2\dd \bx\dd s
&=\left |\int_0^T\!\!\int_{\Omega_{\eps}}e_p^{\eps}\diver\bs\phi\, \dd \bx\dd s\right|\\
&=\left |\int_0^T\!\!\int_{\Omega_{\eps}}\Big (\T^{-1}\vro_f\partial_t\bs e_f^\eps - 
\bs r_f^{\eps}+f_3^{\eps}\bs e_3\Big )\cdot\bs\phi\, \dd \bx\dd s \right .
\\
&\qquad
\left .
-2\eta \int_0^T\!\!\int_{\Omega_{\eps}}\sym\nabla\bs e_f^\eps:\sym\nabla\bs\phi\,\dd \bx\dd s
\right |\\
&\leq C\eps^{1/2}\left(h^{3\gamma/2 + 3/2-\kappa/2} + \eps + h^{\min\{\gamma/2,(2\gamma-\kappa+2)_+\}}\right)\|e_p^{\eps}-\pi_e^{\eps}\|_{L^2(0,T;L^2(\Omega_{\eps})}.
\end{align*}
Here we used the higher-order energy estimate of Corollary \ref{TimeRefEst} 
to control the time derivatives, definition of the fluid residual term $\res_f^\eps$ and the estimate
of Proposition \ref{prop:EE}. 

To estimate the mean value term 
$\pi_e^{\eps}$ we follow the same steps as in the proof of estimate \eqref{ineq:pressure} 
with $\zeta=\pi_e^{\eps}$.
However, we do not gain anything in comparison to the a priori estimates because we have not 
derived higher-order estimates for $\partial_{tt}^2\bs e_s^h$. 
Therefore, for $\kappa\geq2\gamma+1$ we proved the following 
error estimates for the pressure:
\begin{equation}\label{error:pressure}
\|p^{\eps}-\paa^{\eps}\|_{L^2(0,T;L^2(\Omega_{\eps})} \leq 
C\eps^{1/2}h^{\min\{\gamma/2,2\gamma-\kappa+2\}} + C\eps h^{-1-\tau}
\leq C\eps^{1/2}h^{\min\{\gamma/2,(2\gamma-\kappa+2)_+\}}\,.
\end{equation}
This finishes the proof of Theorem \ref{tm:EE}.

\appendix

\section{Griso decomposition and Korn inequality on thin domain}\label{app:A}
The following result is directly from \cite[Theorem 2.3]{Gri05}, tailored to the specific 
boundary conditions and geometry considered in this paper.
\begin{theorem}
Let $h>0$, then every $\bu^h\in V_S(\Omega_h)$ can be decomposed as
\begin{equation}\label{app:Griso_dec}
\bu^h(x) = \bs w^h(x') + (x_3 - h/2)\bs e_3\times\bs r^h(x') + \tilde\bu^h(x)\,,\quad (x',x_3)\in\Omega_h\,,
\end{equation}
or written componentwise
\begin{align*}
u_1^h(x) &= w_1^h(x') + (x_3-h/2)r_2^h(x') + \tilde u_1^h(x)\,,\\
u_2^h(x) &= w_2^h(x') - (x_3-h/2)r_1^h(x') + \tilde u_2^h(x)\,,\\
u_3^h(x) &= w_3^h(x') + \tilde u_3^h(x)\,,
\end{align*}
where 
\begin{align*}
\bs w^h(x') = \frac{1}{h}\int_0^h \bu^h(x',x_3)\dd x_3\,,\quad
 \bs r^h(x') = \frac{3}{h^3}\int_0^h (x_3 - h/2)\bs e_3\times\bu^h(x',x_3)\dd x_3\,,
\end{align*}
and $\tilde\bu^h\in V_S(\Omega_h)$ is so called \emph{warping} or residual term.
The main part of the decomposition, 
denoted by $\bu_E^h = \bs w^h(x') + (x_3 - h/2)\bs e_3\times\bs r^h(x')$,
is called the \emph{elementary plate displacement}. 
Moreover, the following estimate holds
\begin{equation}\label{app:Griso_estimate}
\|\sym\nabla \bu_E^h\|_{L^2(\Omega_h)}^2
 + \|\nabla\tilde{\bu}^h\|_{L^2(\Omega_h)}^2
 + \frac{1}{h^2}\|\tilde{\bu}^h\|_{L^2(\Omega_h)}^2
 \leq C\|\sym\nabla{\bu}^h\|_{L^2(\Omega_h)}^2\,,
\end{equation}
where $C>0$ is independent of $\bu^h$ and $h$.
\end{theorem}

\begin{theorem}[Korn inequality on thin domains]\label{app:thinKorn}
Let $\omega\subset\R^2$ be Lipschitz domain and $\gamma\subset\pa\omega$ part of its boundary
of positive measure, then there exists a constant $C_K>0$ and $h_0>0$ such that for every $0<h<h_0$
\begin{equation*}
\|(\psi_1,\psi_2,h\psi_3)\|_{H^1(\Omega;\R^3)}^2 \leq C_K\Big (\|(\psi_1,\psi_2,h\psi_3)\|_{L^2(\Omega;\R^3)}^2+\|\sym\nabla_h\bs\psi\|_{L^2(\Omega;\R^9)}^2\Big )\,,
\quad\forall\bs\psi\in H^1(\Omega;\R^3)\,,
\end{equation*}
where $\Omega = \omega\times (0,1)$. The Korn constant
$C_K$ depends only on $\omega$ and $\gamma$.
\end{theorem}
\begin{proof}
The proof follows by the Griso's decomposition of $\bs\psi\in H_\gamma^1(\Omega;\R^3)$ 
(see \cite{Gri05}) and application of the Korn inequality for functions defined on
$\omega$.
\end{proof}


\begin{thebibliography}{11}

\bibitem{ALT10} G.~Avalos, I.~Lasiecka, R.~Triggiani. Higher Regularity of a Coupled Parabolic-Hyperbolic 
Fluid-Structure Interactive System. {\em Georgian Mathematical Journal}, 15 (2000), 403--437.

\bibitem{AvTr07} G.~Avalos, R.~Triggiani. The coupled PDE system arising in fluid/structure interaction.
I. Explicit semigroup generator and its spectral properties. {\em Fluids and waves},
15--54, Contemp. Math., 440, Amer. Math. Soc., Providence, RI, 2007.

\bibitem{AvTr09} G.~Avalos and R.~Triggiani. Semigroup well-posedness in the energy space of a 
parabolic-hyperbolic coupled Stokes-Lam\'e PDE system of fluid-structure interaction. 
{\em Discr.~Contin.~Dyn.~Sys.~Series S}, 2 (2009), 417--447.

\bibitem{BayCha86}
Guy Bayada and Mich\`ele Chambat.
\newblock The transition between the {S}tokes equations and the {R}eynolds
  equation: a mathematical proof.
\newblock {\em Appl. Math. Optim.}, 14(1):73--93, 1986.

\bibitem{BeGr05} J.~Becker, G.~Gr\"un. The thin-film equation: 
Recent advances and some new perspectives. {\em J.~Phys.: Condens.~Matter}, 17 (2005), 291--307.

 \bibitem{Ber98}
 A. Bertozzi. The mathematics of moving contact lines in thin liquid films. 
 {\em Notices Amer.~Math. Soc.}, 45 (1998), 689-697.

\bibitem{BGN14} T.~Bodn\'ar, G.~P.~Galdi, \v S.~Ne\v casova. Fluid-Structure Interaction in Biomedical Applications.  Springer/Birkhouser. 2014.


\bibitem{Bol63} V.~V.~Bolotin. Nonconservative problems of the theory of elastic stability.
Pergamon Press, London, 1963.

\bibitem{BuMu20B} M.~Bukal, B.~Muha. A review on rigorous derivation of reduced models for fluid-structure interaction systems.
To appear in {\em Waves in Flows}, Eds.~T.~Bodn\'ar, G.~P.~Galdi, and \v S.~Ne\v casov\'a, Birkh\"auser, Cham, 2020.

\bibitem{BuMu20} M.~Bukal, B.~Muha. Justification of a nonlinear sixth-order thin-film 
equation as the reduced model for a fluid - structure interaction problem. 
{\em In preparation} (2020).

\bibitem{BuDe05} A.~P.~Bunger, and E.~Detournay. Asymptotic solution for a penny-shaped near-surface hydraulic fracture. 
{Engin.~Fracture Mech.} 72 (2005), 2468--2486.

\bibitem{BCMG16} M.~Buka\v c, S.~\v Cani\'c, B.~Muha and R.~Glowinski. 
An Operator Splitting Approach to the Solution of Fluid-Structure Interaction Problems in Hemodynamics, 
{\em in Splitting Methods in Communication and Imaging, Science and Engineering Eds. R. Glowinski, S. Osher, and W. Yin},  New York, Springer, 2016.

\bibitem{Cia97} P.~G.~Ciarlet. Mathematical Elasticity. Vol.~II: Theory of Plates. North-Holland Publishing 
Co, Amsterdam, 1997.

\bibitem{Cia88} P.~G.~Ciarlet. Mathematical Elasticity. Vol.~I: Three-dimensional elasticity. North-Holland Publishing Co, Amsterdam, 1988.


\bibitem{CDG18} D.~Cioranescu, A.~Damlamian, G.~Griso. The Periodic Unfolding Method:
Theory and Applications to Partial Differential Problems, Series in Contemporary Mathematics 3. 
Springer, 2018.

\bibitem{CDEM}
Antonin Chambolle, Beno{\^{\i}}t Desjardins, Maria~J. Esteban, and C{\'e}line
  Grandmont.
\newblock Existence of weak solutions for the unsteady interaction of a viscous
  fluid with an elastic plate.
\newblock {\em J. Math. Fluid Mech.}, 7(3):368--404, 2005.

\bibitem{Chu14}
  Chueshov, Igor, Dynamics of a nonlinear elastic plate interacting with a linearized compressible viscous fluid, {\em Nonlinear Anal.}, 95 (2014), 650--665
   

\bibitem{CDLW16} I.~Chueshov, E.~H.~Dowell, I.~Lasiecka, and J.~T.~Webster. 
Mathematical aeroelasticity: A survey. {\em Journal MESA} 7 (2016), 5--29.

\bibitem{CanMik03}
S.~\v{C}ani\'{c} and A.~Mikeli\'{c}.
\newblock Effective equations modeling the flow of a viscous incompressible
  fluid through a long elastic tube arising in the study of blood flow through
  small arteries.
\newblock {\em SIAM J. Appl. Dyn. Syst.}, 2(3):431--463, 2003.

\bibitem{CSS2}
D.~Coutand and S.~Shkoller.
\newblock The interaction between quasilinear elastodynamics and the
  {N}avier-{S}tokes equations.
\newblock {\em Arch. Ration. Mech. Anal.}, 179(3):303--352, 2006.

\bibitem{CuMP18} A.~\'Curkovi\'c and E.~Maru\v si\'c-Paloka. Asymptotic analysis of a thin 
fluid layer-elastic plate interaction problem. {\em Applicable analysis} 98 (2019), 2118--2143.

\bibitem{DJBH08} S.~B.~Das, I.~Joughin, M.~Behn, I.~Howat, M.~A.~King, D.~Lizarralde, M.~P.~Bhatia. 
Fracture propagation to the base of the Greenland ice sheet during supraglacial lake
drainage. {\em Science} 320 (2008), 778--781.

\bibitem{DaFin06} R.~Daw and J.~Finkelstein. Lab on a chip. {\em Nature Insight} 442 (2006),
367--418.


 

\bibitem{Dow15}Earl H.~Dowell. A modern course in aeroelasticity. Volume 217 of the Solid Mechanics and
its Applications book series. Springer, 2015.

\bibitem{DGHL03} Q.~Du, M.~D.~Gunzburger, L.~S.~Hou, and J.~Lee. Analysis of a linear fluid-structure
interaction problem. {Discr.~Cont.~Dyn.~Sys.} 9 (2003), 633-650.

\bibitem{DuMP00}A.~Duvnjak, E.~Maru\v si\'c-Paloka.
Derivation of the Reynolds equation for lubrication of a rotating shaft. {\em Arch.~Math.} 36 (2000), 239-253.


\bibitem{Galdi} G.~P.~Galdi. An Introduction to the Navier-Stokes Initial-Boundary Value Problem. 
In: Galdi G.P., Heywood J.G., Rannacher R. (eds) Fundamental Directions in Mathematical Fluid Mechanics. 
Advances in Mathematical Fluid Mechanics. Birkh\"auser, Basel, 2000.

\bibitem{galdi2011introduction}
G.~P.~Galdi. An Introduction to the Mathematical Theory of the Navier-Stokes
  Equations: Steady-State Problems. Springer, 2011. 



\bibitem{Gri05} G.~Griso. Asymptotic behavior of structures made of plates. {\em Anal.~Appl.}, 3 (2005), 325--356.  

\bibitem{Des81} P.~Destuynder. Comparaison entre les modeles tridimensionnels et bidimensionnels de plaques en \'elasticit\'e.
{\em ESAIM: Mathematical Modelling and Numerical Analysis} 15 (1981), 331--369.

\bibitem{GilTru} Gilbarg, David and Trudinger, Neil S., Elliptic partial differential equations of second order,Springer-Verlag, Berlin-New York, 1977, x+401


\bibitem{HBB13} I.~J.~Hewit, N.~J.~Balmforth, and J.~R.~de Bruyn. Elastic-plated gravity currents. 
{\em Euro.~Jnl.~of Applied Mathematics} 26 (2015), 1--31.

\bibitem{HHS08} M.~Heil, A.~L.~Hazel, and J.~A.~Smith. The mechanics of airway closure. {\em Respiratory
Physiology \& Neurobiology} 163 (2008), 214--221.

\bibitem{HoMa04} A.~E.~Hosoi, and L.~Mahadevan. Peeling, healing and bursting in a lubricated elastic sheet.
{\em Phys.~Rev.~Lett.} 93 (2004).

\bibitem{HuSo02} R.~Huang, and Z.~Suo. Wrinkling of a compressed elastic film on a viscous layer. {J.~Appl.~Phys.} 91 (2002), 1135--1142.

\bibitem{KKLTTW18} B.~Kaltenbacher, I.~Kukavica, I.~Lasiecka, R.~Triggiani, A.~Tuffaha, and J.~T.~Webster.
Mathematical Theory of Evolutionary Fluid-Flow Structure Interactions. Birkhuser, 2018.

\bibitem{King89} J.~R.~King. The isolation oxidation of silicon the reaction-controlled case. {SIAM J.~Appl.~Math.} 49 (1989), 1064--1080.

\bibitem{LBS05} E.~Lauga, M.~P.~Brenner and H.~A.~Stone.
Microfluidics: The No-Slip Boundary Condition. In {\em Handbook of Experimental Fluid Dynamics}
Eds.~ J.~Foss, C.~Tropea and A.~Yarin, Springer, New-York (2005).

\bibitem{LeMu11} M.~Lewicka, S.~M\"uller. The uniform Korn-Poincar\'e inequality in thin domains. 
{\em Annales de l'Institut Henri Poincare (C) Non Linear Analysis} 28 (2011),  443--469.


\bibitem{LionsMagenes} Lions, J.-L. and Magenes, E., Non-homogeneous boundary value problems and applications. Vol. I \& II, Springer-Verlag, New York-Heidelberg, 1972

\bibitem{LPN13} J.~R.~Lister, G.~G.~Peng, and J.~A.~Neufeld. Spread of a viscous fluid beneath an elastic
sheet. {\em Phys.~Rev.~Lett.} 111 (15) (2013).


\bibitem{MaPa01} E.~Maru\v si\'c-Paloka. The effects of flexion and torsion on a fluid flow through a curved pipe. {\em Appl.
Math. Optim.}, 44 (2001), 245-272.
\bibitem{Mic11} C.~Michaut. Dynamics of magmatic intrusions in the upper crust: Theory and applications
to laccoliths on Earth and the Moon. {\em J.~Geophys.~Res.} 116 (2011).

\bibitem{MikGuCan07}
Andro Mikeli\'{c}, Giovanna Guidoboni, and Sun\v{c}ica \v{C}ani\'{c}.
\newblock Fluid-structure interaction in a pre-stressed tube with thick elastic
  walls. {I}. {T}he stationary {S}tokes problem.
\newblock {\em Netw. Heterog. Media}, 2(3):397--423, 2007.


\bibitem{SunBorMulti}
B.~Muha and S.~{\v{C}}ani{\'c}.
\newblock Existence of a solution to a fluid--multi-layered-structure
  interaction problem.
\newblock {\em J. Differential Equations}, 256(2):658--706, 2014.

\bibitem{NazPil90}
S.~A. Nazarov and K.~I. Piletskas.
\newblock The {R}eynolds flow of a fluid in a thin three-dimensional channel.
\newblock {\em Litovsk. Mat. Sb.}, 30(4):772--783, 1990.

\bibitem{ODB97} A. Oron, S. H. Davis, S. G. Bankoff. Long-scale evolution of thin liquid films. {\em Rev. Mod. Phys.} 69 (1997), 931-980.

\bibitem{PaSt06} G.~P.~Panasenko, R.~Stavre. Asymptotic analysis of a periodic flow in a thin channel 
with visco-elastic wall. {\em J.~Math.~Pures Appl.}~85 (2006), 558-579.

\bibitem{PaSt14} G.~P.~Panasenko, R.~Stavre. Asymptotic analysis of a viscous fluid-thin plate 
interaction: Periodic flow. {\em Mathematical Models and Methods in Applied Sciences} 24 (2014), 1781-1822.

\bibitem{PaSt20A} G.~P.~Panasenko, R.~Stavre. Viscous fluid-thin elastic plate interaction: asymptotic analysis with respect to the rigidity
and density of the plate. {\em Appl.~Math.~Optim.} 81 (2020), 141-194.

\bibitem{PaSt20B} G.~P.~Panasenko, R.~Stavre. Three Dimensional Asymptotic Analysis of an Axisymmetric 
Flow in a Thin Tube with Thin Stiff Elastic Wall. {\em J.~Math.~Fluid Mech.} 22, 20 (2020). 
https://doi.org/10.1007/s00021-020-0484-8

\bibitem{PIHJ12} D.~Pihler-Puzovi\'c, P.~Illien, M.~Heil, and A.~Juel. Suppression of complex fingerlike patterns
at the interface between air and a viscous fluid by elastic membranes. {\em Phys. Rev.
Lett.} 108 (2012).

\bibitem{PJH14} D.~Pihler-Puzovi\'c, A.~Juel and M.~Heil. The interaction between viscous fingering and wrinkling in elastic-walled Hele-Shaw cells. 
{\em Phys.~Fluids} (in press) (2014).  

\bibitem{SSA04} H.~A.~Stone, A.~D.~Stroock, A.~Ajdari. 
Engineering Flows in Small Devices: Microfluidics Toward a Lab-on-a-Chip. 
{\em Annual Review of Fluid Mechanics} 36 (2004), 381-411.

\bibitem{Sze12} A.~Z.~Szeri. Fluid Film Lubrication. Cambridge University Press, Cambridge, 2012.

\bibitem{TamCanMik05}
J.~Tamba\v{c}a, S.~\v{C}ani\'{c}, and A.~Mikeli\'{c}.    
\newblock Effective model of the fluid flow through elastic tube with variable
  radius.
\newblock In {\em X{I}. {M}athematikertreffen {Z}agreb-{G}raz}, volume 348 of
  {\em Grazer Math. Ber.}, pages 91--112. Karl-Franzens-Univ. Graz, Graz, 2005.

\bibitem{TaVe12} M.~Taroni, and D.~Vella. Multiple equilibria in a simple elastocapillary system. {\em J.~Fluid
Mech.} 712 (2012), 273--294.  

\bibitem{Tit94} I.~Titze. Principles of voice production. Prentice Hall, New York, 1994.

\bibitem{TsRi12} V.~C.~Tsai, and J.~R.~Rice. Modeling turbulent hydraulic fracture near a free surface. {\em J.~App.~Mech.} 79 (2012).


\bibitem{YGJ17} A.~Yenduri, R.~Ghoshal, and R.~K.~Jaiman. 
A new partitioned staggered scheme for flexible multibody interactions with strong inertial effects. 
{\em Computer Methods in Applied Mechanics and Engineering} 315 (2017), 316-347.
\end{thebibliography}
\end{document}